\documentclass[12pt,reqno,a4paper]{amsart}
\usepackage{
    amsmath,  amsfonts, amssymb,  amsthm,   amscd, 
    gensymb,  graphicx, comment,  etoolbox, url,
    booktabs, stackrel, mathtools,enumitem, mathdots,  microtype, lmodern,    mathrsfs, graphicx, tikz,  longtable,tabularx, float, tikz, pst-node, tikz-cd, multirow, tabularx, amscd,  bm, array, makecell, diagbox, booktabs,ragged2e, caption, subcaption }
\usepackage{makecell,slashbox}
\usepackage{xcolor}
\usepackage[all]{xy}
\usepackage[pagebackref=true, colorlinks=true, linkcolor=blue, citecolor=blue, urlcolor=blue, breaklinks=true]{hyperref}
\usepackage[utf8]{inputenc}
\usepackage{microtype, fullpage, wrapfig,textcomp,mathrsfs,csquotes,fbb}
\usepackage[colorlinks=true, linkcolor=blue, citecolor=blue, urlcolor=blue, breaklinks=true]{hyperref}

\usepackage[capitalise]{cleveref}
\usepackage{todonotes}
\usetikzlibrary{positioning}
\usetikzlibrary{shapes,arrows.meta,calc}
\usetikzlibrary{arrows}

\newtheorem{theorem}{Theorem}[section]

\newtheorem{corollary}[theorem] {Corollary}
\newtheorem{lemma}[theorem]{Lemma}

\newtheorem{proposition}[theorem]{Proposition}

\newtheorem*{theorem*}{Theorem}
\newtheorem*{conjecture*}{Conjecture}
\theoremstyle{definition}
\newtheorem{definition}[theorem]{Definition}
\newtheorem{example}[theorem]{Example}
\newtheorem{remark}[theorem]{Remark}

\newtheorem*{notation*}{Notation}

\newcommand{\TC}{\mathrm{TC}}

\newcommand{\cl}{\mathrm{cl}}
\newcommand{\zcl}{\mathrm{zcl}}
\newcommand{\cat}{\mathrm{cat}}

\newcommand{\sct}{\mathrm{secat}}

\def\gsc#1#2{\mathrm{secat}_{#1}(#2)}

\def\ra{\rightarrow}
\def\xra{\xrightarrow}
\def\xla{\xleftarrow}

\makeatletter
\def\@tocline#1#2#3#4#5#6#7{\relax
  \ifnum #1>\c@tocdepth 
  \else
    \par \addpenalty\@secpenalty\addvspace{#2}%
    \begingroup \hyphenpenalty\@M
    \@ifempty{#4}{%
      \@tempdima\csname r@tocindent\number#1\endcsname\relax
    }{%
      \@tempdima#4\relax
    }%
    \parindent\z@ \leftskip#3\relax \advance\leftskip\@tempdima\relax
    \rightskip\@pnumwidth plus4em \parfillskip-\@pnumwidth
    #5\leavevmode\hskip-\@tempdima
      \ifcase #1
       \or\or \hskip 1em \or \hskip 2em \else \hskip 3em \fi%
      #6\nobreak\relax
    \hfill\hbox to\@pnumwidth{\@tocpagenum{#7}}\par
    \nobreak
    \endgroup
  \fi}
\makeatother

\newcolumntype{x}[1]{>{\centering\arraybackslash}p{#1}}

\begin{document}
\title[]{Equivariant relative sectional category \\ and induced invariants}
\author[N. Daundkar]{Navnath Daundkar}
\address{Indian Institute of Technology Madras, Chennai, India.}
\email{navnath@iitm.ac.in}
\author{Rekha Santhanam}
\address{Indian Institute of Technology Bombay, Mumbai, India}
\email{reksan@iitb.ac.in}
\author{Bittu Singh}
\address{Indian Institute of Technology Bombay, Mumbai, India}
\email{22d0781@iitb.ac.in}

\thanks{}

\begin{abstract} 
The relative sectional category, introduced by Gonz\'alez, Grant, and Vandembroucq for fibrations and later extended by Garc\'ia-Calcines to arbitrary maps, provides a common framework encompassing several numerical homotopy invariants, including the Lusternik--Schnirelmann category, the topological complexity of a map, and homotopic distance. In this paper, we introduce and study the equivariant analogue of the relative sectional category for $G$-maps. We establish its fundamental homotopy-theoretic properties, including comparison, product, and composition inequalities, as well as its behavior under changes of domain and codomain. As applications, we introduce and investigate equivariant analogues of the topological complexity of a map, in the sense of Scott and Murillo--Wu, and the equivariant Lusternik--Schnirelmann category of a map. Several examples are provided to illustrate the theory and demonstrate that these invariants extend the corresponding classical equivariant notions.
\end{abstract}

\keywords{Relative sectional category, LS category, Topological complexity, Homotopic distance, equivariant sectional category}
\subjclass[2020]{55M30, 55M99, 55R91}
\maketitle

\section{Introduction}
Topological complexity, introduced by Farber \cite{FarberTC}, is a numerical homotopy invariant that measures the complexity of motion planning algorithms in robotics. It is defined as the sectional category of the free path space fibration and has inspired numerous generalizations over the past two decades.
A classical invariant closely related to topological complexity is the Lusternik--Schnirelmann (LS) category, introduced by Lusternik and Schnirelmann in \cite{LScat}. Roughly speaking, the LS-category of a space is the smallest number of contractible pieces into which the space can be decomposed.

A common framework underlying many of these invariants is the \emph{sectional category} of a map, introduced by \v{S}varc \cite{Sva} for fibrations and later extended by Bernstein and Ganea \cite{secat} to arbitrary maps. Besides recovering classical invariants such as the Lusternik--Schnirelmann category and topological complexity, sectional category provides a unified approach to a wide range of numerical homotopy invariants.

An important extension of this notion is the \emph{relative sectional category}, introduced by Gonz\'alez, Grant, and Vandembroucq \cite{Hopf_inv_Gon_Luc_Grant} for fibrations and later generalized by Garc\'ia-Calcines \cite{Calcines_relTC} to arbitrary maps. This invariant unifies several homotopy invariants of maps, including the LS-category of a map, Scott's topological complexity of a map, and the homotopic distance of Mac\'ias-Virg\'os and Mosquera-Lois \cite{MACÍAS–VIRGÓS_MOSQUERA–LOIS_2022}. Moreover, it enjoys many of the formal properties of the classical sectional category, including product, composition, and comparison inequalities.

The presence of symmetries naturally led to the development of equivariant versions of these invariants. Fadell \cite{Fadelleqcat} introduced the equivariant LS-category, which was subsequently studied by Marzantowicz \cite{Eqlscategory}. Later, Colman and Grant \cite{Colman_Grant_Eqtc} introduced equivariant topological complexity by defining it as the equivariant sectional category of equivariant fibrations. More recently, the first author and Garc\'ia-Calcines developed the theory of equivariant homotopic distance \cite{D-GC}, extending the homotopic distance of maps to the equivariant setting and establishing several of its fundamental properties. These developments significantly broadened the scope of the LS-category, topological complexity, and related homotopy invariants, extending their applicability to the setting of $G$-spaces equipped with actions of compact Lie groups.

In this paper, we introduce the notion of \emph{equivariant relative sectional category} and establish its fundamental homotopy-theoretic properties. As applications, we introduce and study equivariant analogues of several important invariants arising from the relative sectional category, including the equivariant topological complexity of a map and the equivariant LS-category of a map.

The paper is organized as follows. In Section~\ref{sec: prelim}, we recall the background material required throughout the paper. We begin with a brief review of equivariant homotopy pullbacks, followed by the notions of equivariant sectional category and relative sectional category, together with the fundamental properties that will be used in the sequel. These preliminaries provide the framework for developing the theory of equivariant relative sectional category.

Section~\ref{sec: eq rel secat} introduces the equivariant relative sectional category and establishes its basic properties. We investigate its behaviour under natural constructions and prove a Ganea-type characterization. In Section~\ref{sec: eq tc of eq maps}, we apply this theory to the study of the equivariant topological complexity of equivariant maps. We first introduce the equivariant Lusternik--Schnirelmann category of equivariant maps and then study two notions of equivariant topological complexity of equivariant maps, namely Scott's version and the Murillo--Wu version. In \ref{subsec: eq tc MW version}, we prove that the Murillo--Wu version is equivalent to Scott's version. Finally, Section~\ref{subsec: examples} concludes the paper with several examples illustrating the theory and highlighting the relationship between these invariants.
\section{Preliminaries}\label{sec: prelim}
Let  $G$ be a compact Hausdorff topological group, and a $G$-space is a Hausdorff space $X$  equipped with a continuous left action of $G$.
 Let $X,Y$ be $G$-spaces, then the product $X\times Y$ gets a canonical $G$-action defined by $g(x,y)=(gx,gy)$ for all $(x,y) \in X \times Y$,  called the diagonal action of $G$ on $X\times Y$.
 
 The \emph{stabilizer/isotropy group} $G_x$  is defined by the closed subgroup of $G$ containing the elements $g\in G$ such that $gx=x.$ The \emph{orbit } $\mathcal O_x$ set is defined as  $G\cdot x$, that is, the subspace of $X$ containing the elements $gx$ for all $g\in G.$ The orbit-stabilizer theorem provides a homeomorphism $G/G_x\cong \mathcal O_x.$

 For a subgroup $H\leq G$, the \emph{$H$-fixed point set} $X^H$ of $X$ is defined to be the subspace of all points $x\in X$ such that $hx=x$ for all $h\in H.$ 

Let $f,h :X\ra Y$ be two $G$-maps.
The maps $f,h$ are said to be \emph{$G$-homotopic} if there exists a $G$-map $F\colon X\times I\ra Y$ such that $F(-,0)=f$ and $F(-,1)=h$; where the unit interval $I$ is trivial $G$-space and $G$-acts on $X\times I$ diagonally.

For a $G$-space $B$, the corresponding path space $B^I$ admits a $G$ action via $(g\cdot \gamma)(t)=g(\gamma(t))$. 
The \emph{standard $G$-homotopy pullback} for $G$-maps $f\colon K\to B$ and $p\colon E \to B$, is a square
\[\begin{tikzcd}
	{P_{f,p}} & E \\
	K & {B,}
	\arrow["{\tilde f}", from=1-1, to=1-2]
	\arrow["{\tilde p}"', from=1-1, to=2-1]
	\arrow["p", from=1-2, to=2-2]
	\arrow["f"', from=2-1, to=2-2]
\end{tikzcd}\]
where $$P_{f,p}=\{(k,\gamma, e)\in K \times B^I \times E \mid \gamma(0)=f(k), \gamma(1)=p(e)\},$$ equipped with 
the diagonal $G$-action. The maps $\tilde f$ and $\tilde p$ are defined as $\tilde f(k,\gamma,e)=e$ and $\tilde p(k,\gamma ,e)=k$.

We recall the well-known result establishing when a $G$-homotopy pullback is a $G$-pullback.
\begin{proposition}\label{prop: hpb=pb}
    If $f$ or $p$ is a $G$-fibration, then $G$-homotopy pullback square is a $G$-pullback square.
\end{proposition}
\subsection{Equivariant sectional category}
The sectional category of a Hurewicz fibration was originally introduced by \v{S}varc in  \cite{Sva}. Later, Colman and Grant \cite{Colman_Grant_Eqtc} extended this concept to the equivariant setting, defining the equivariant sectional category. 
\begin{definition}[{\cite[Definition 4.1]{Colman_Grant_Eqtc}}]\label{def:eqsecat}
The \emph{equivariant sectional category} of a $G$-map $p\colon E\to B$, denoted $\sct_G(p)$, is the least nonnegative integer $n$ such that there is a cover of $B$ by $G$-invariant open subsets $\{U_0, \ldots, U_n\}$ and, for each $i$, a local $G$-homotopy section $s_i \colon U_i \to E$ of $p$; that is, a $G$-map $s_i$  such that the following diagram commutes up to $G$-homotopy
$$\xymatrix{
{U_i} \ar@{^(->}[rr]^{\iota_{U_i}} \ar[dr]_{s_i} & & {B} \\
 & {E}, \ar[ur]_p &  }$$
\noindent where $\iota_{U_i}\colon  U_i \hookrightarrow B$ is the inclusion map. If no such integer exists, we set $\sct_G(p):=\infty.$ 
\end{definition}
Another useful feature, which will play a role in later arguments, is the subadditivity of the equivariant sectional category.  Its proof can be found in \cite[Proposition 2.11]{arora2024equivariant}.
\begin{proposition}\label{sub}
Let $p_i\colon E_i\to B_i$ be a $G$-fibration, for $i\in \{1,2\}$.  
Then the product map
\[
p_1\times p_2\colon E_1\times E_2 \longrightarrow B_1\times B_2
\]
is a $G$-fibration, where $G$ acts diagonally on both $E_1\times E_2$ and $B_1\times B_2$.  
Moreover, if $B_1\times B_2$ is completely normal and $B_1,B_2$ are Hausdorff, then
\[
\sct _G(p_1\times p_2)\leq \sct _G(p_1)+\sct _G(p_2).
\]
\end{proposition}
\bigskip
Fadell \cite{Fadelleqcat} introduced the notion of the $G$-equivariant Lusternik-Schnirelmann (LS) category for $G$-spaces. 
This concept was further developed by Marzantowicz \cite{Eqlscategory}, Clapp and Puppe \cite{Clapp-Puppe}, Colman \cite{Colmaneqcat}, and Angel, Colman, Grant, and Oprea \cite{Angel-Colman-Grant-Oprea-Moritainv-eqcat}. 
This homotopy invariant of a $G$-space $X$ is denoted by $\cat_G(X)$.  

Before stating the definition, we recall the notion of a $G$-categorical set. 
A $G$-invariant open subset $U \subseteq X$ is called \emph{$G$-categorical} if the inclusion $\iota_{U}\colon U \hookrightarrow X$ is $G$-homotopic to a $G$-map whose image lies entirely within a single orbit.  
\begin{definition}
For a $G$-space $X$, the \emph{$G$-equivariant LS category} $\cat_G(X)$ is the least nonnegative integer $n$ such that $X$ can be covered by $n+1$ $G$-categorical sets. 
If no such integer exists, we set $\cat_G(X):= \infty$.
\end{definition}
Although standard, we briefly recall some basic facts about fixed points under closed subgroups, 
both to fix notation and for later use. 
Let $H$ be a closed subgroup of $G$, and let $X$ be a $G$-space. 
The set of $H$-fixed points of $X$ is denoted by $X^H$ and defined as
\[
X^H := \{x \in X \mid h \cdot x = x \text{ for all } h \in H\}.
\]
If $f \colon X \to Y$ is a $G$-map, then $f(X^H) \subseteq Y^H$, so the restriction of $f$ to $X^H$ induces a map 
\[
f^H \colon X^H \to Y^H.
\]
The fixed-point construction also allows us to formulate a natural equivariant analogue of connectedness:
\begin{definition}
A $G$-space $X$ is said to be \emph{$G$-connected} if, for every closed subgroup $H$ of $G$, the fixed-point set $X^H$ is path-connected.
\end{definition}
This notion provides a useful bridge between equivariant sectional category and equivariant LS category:
\begin{remark}\label{fund}
If $x_0\in X$ is a fixed point of the $G$-action and $X$ is $G$-connected, then the equivariant sectional category of the inclusion 
\(\iota\colon \{x_0\} \to X\) satisfies
\[
\sct_G(\iota\colon \{x_0\} \to X) = \cat_G(X);
\]
see \cite[Corollary 4.7]{Colman_Grant_Eqtc}.  
Moreover, when $G$ is the trivial group, this recovers the classical invariant, i.e.\ \(\cat_G(X)=\cat(X)\).
\end{remark}
We now recall two standard bounds for the equivariant sectional category.  
First, we record a slightly modified version of the equivariant homotopy dimension-connectivity upper bound, as stated in \cite{Symm_grant,arora2024equivariant}. 
Throughout, a \emph{Serre $G$-fibration} will mean a $G$-map $p \colon E \to B$ satisfying the $G$-homotopy lifting property with respect to all $G$-CW complexes. 
Clearly, every  $G$-fibration is in particular a Serre $G$-fibration.
\begin{proposition}\label{connect}
Let $p \colon E \to B$ be a Serre $G$-fibration, and suppose that $B$ is a $G$-CW complex with $\dim(B) \geq 2$. 
Assume that for every closed subgroup $H \subseteq G$, the fixed-point map $p^H \colon E^H \to B^H$ is an $m$-equivalence. 
Then
\[
\sct_G(p) \;<\; \frac{\mathrm{hdim}_G(B) + 1}{m + 1}.
\]
Here $\mathrm{hdim}_G(B)$ denotes the minimal dimension $\dim(B')$ such that $B'$ is a $G$-CW complex $G$-homotopy equivalent to $B$.
\end{proposition}
Finally, we state the classical cohomological lower bound for the equivariant sectional category 
by Arora and Daundkar in \cite{arora2024equivariant}. 
We begin by fixing notation for Borel cohomology. 
Let $EG \to BG$ be the universal principal $G$-bundle. 
For a $G$-space $X$, its homotopy orbit space is
\[
X^h_G := EG \times_G X,
\]
and its Borel cohomology with coefficients in a commutative ring $R$ is
\[
H^\ast_G(X;R) := H^\ast(X^h_G;R).
\]
If $f \colon X \to Y$ is a $G$-map, we denote by $f^h_G \colon X^h_G \to Y^h_G$ the induced map between the corresponding homotopy orbit spaces.
\begin{proposition}\label{Gcohom}
Let $p \colon E \to B$ be a $G$-map. 
Suppose there exist cohomology classes $u_1,\dots,u_k \in H^\ast_G(B;R)$ such that
\[
(p^h_G)^\ast(u_1) = \cdots = (p^h_G)^\ast(u_k) = 0,
\quad\text{and}\quad
u_1 \smile \cdots \smile u_k \neq 0.
\]
Then $\sct_G(p) \geq k$.
\end{proposition}
Together with the dimension-connectivity upper bound, this cohomological lower bound will serve as a key tool in the sections that follow.
\subsection{Relative sectional category}
We recall the following definition of the relative sectional category and the invariants derived from it.
\begin{definition}[{\cite{Hopf_inv_Gon_Luc_Grant, Calcines_relTC}}]
Let $K\xra{f} B\xla p E $ be continuous maps. The \emph{sectional category} of $p$ with respect to $f$, denoted by $\sct_f(p)$, is the least integer $m \leq \infty$ such that $K$ admits an open cover $\{U_i\}_{i=0}^m$ with the property that, for each $i$, there exists a map $s_i \colon U_i \to E$ satisfying
\[
p \circ s_i \simeq f|_{U_i}.
\]
\end{definition}
The notion of topological complexity of maps was introduced by Scott in \cite{Sc}. Here we define it in the sense of Garc\'ia-Calcines as a relative sectional category of a particular map.
\begin{definition}[{\cite{Sc}}]
    Let $f\colon X\to Y$ be a continuous map. Topological complexity of the map $f$ is defined by $\TC(f)=\sct_{f\times f}(\pi_Y).$
\end{definition}
There is another approach given by Murillo and Wu in \cite{M-W} to the topological complexity of maps. Again, Garc\'ia-Calcines observed \cite{Calcines_relTC} that this notion can be defined as follows:
\begin{definition}  Let $f\colon X\to Y$ be a continuous map. Topological complexity of the map $f$ is defined by $\TC^{MW}(f)=\sct_{f\times f}((f\times f)\circ\pi_X).$
\end{definition}
Note that these two notions of topological complexity maps are equivalent, as shown by Scott in \cite[Theorem 3.4]{Sc}.

\section{Equivariant relative sectional category}\label{sec: eq rel secat}
In this section, we introduce and study the equivariant analogue of relative sectional category. 
\begin{definition}
    Let $p\colon E\to B$  and $f\colon K\to B$ be $G$-maps. The \emph{equivariant sectional category of $p$ relative to $f$}, denoted by $\sct_{G,f}(p)$, is defined as the smallest non-negative integer $n$, such that there is a $G$-invariant open cover $\{U_0,\dots,U_n\}$ of $K$ with $G$-maps $s_i\colon U_i\to E$ satisfying $p\circ s_i\simeq_G f|_{U_i}$. In other words,  the following diagram commutes up to $G$-homotopy:
\[\begin{tikzcd}
	& E \\
	{U_i} & B.
	\arrow["p", from=1-2, to=2-2]
	\arrow["{s_i}", dashed, from=2-1, to=1-2]
	\arrow["{f|_{U_i}}"', from=2-1, to=2-2]
\end{tikzcd}\]
\end{definition}
We refer to $\sct_{G,f}(p)$ as the \emph{equivariant relative sectional category of $p$ with respect to $f$}. A $G$-map $s_i\colon U_i\to E$ satisfying $p\circ s_i\simeq_G f|_{U_i}$ is called a \emph{$G$-equivariant local section  of $p$ relative to $f$} and the corresponding $G$-invariant open cover $K$ is called \emph{$G$-invariant open cover of $K$ relative to $f$}.

When $p\colon E\to B$ is a $G$-fibration, the requirement of having a $G$-map $s_i\colon U_i\to E$ such that $ps_i\simeq_G f|_{U_i}$ is equivalent to having a strict $G$-section $s'_i\colon U_i\to E$ satisfying $ps'_i= f|_{U_i}$. This equivalence follows directly from the $G$-homotopy lifting property of $p$. Therefore, we may use either formulation interchangeably whenever $p$ is a $G$-fibration.

\begin{remark}  We note the following.
    \begin{enumerate}
        \item If $G$ acts trivially on $K, E$ and $B$, then $\sct_{G,f}(p)=\sct_{f}(p)$, which is the relative sectional category in the sense of  Garc\'ia-Calcines \cite{Calcines_relTC}.
        
        \item If $f=\mathrm{id}_B$, then $\sct_{G,f}(p)=\sct_G(p)$.

        \item We always have $\sct_{f}(p)\leq\sct_{G,f}(p)$.
    \end{enumerate}
\end{remark}

\begin{proposition}\label{prop: rel secat=secat hpb}
 Let $p\colon E\to B$  and $f\colon K\to B$ be $G$-maps. Suppose we have the following $G$-homotopy pullback diagram:
 \[ \begin{tikzcd}
P_{f,p} \arrow{r}{\tilde f} \arrow[swap]{d}{{\tilde p}} & E \arrow{d}{p} \\%
K \arrow{r}{f}&B. 
\end{tikzcd}
\]
Then $\sct_{G,f}(p)=\sct_G(\tilde{p})$.
\end{proposition}
\begin{proof}
    Let $\iota_U\colon U \hookrightarrow{K}$ be a $G$-invariant open set admitting a $G$-map $s\colon U \ra E$ such that $p \circ s \simeq_G f|_{U}$. By the property of the $G$-homotopy pullback, we have a $G$-map $\lambda\colon U 
\ra P_{f,p}$ such that $\tilde p\lambda \simeq_G \iota_U.$ This concludes $\sct_G(\tilde{p}) \leq \sct_{G,f}(p)$.

For the reverse inequality, we observe that if such a $G$-equivariant local section $\lambda$ of $p$ exists over $U$ (meaning $\tilde{p} \circ \lambda \simeq_G \iota_U$), we can construct the required map for $p$ by choosing $s = \tilde{f} \circ \lambda$. Thus, $\sct_{G,f}(p) \leq \sct_G(\tilde{p})$. The desired equality follows from using the above arguments for a $G$-invariant open cover.
\end{proof}

The following corollary follows immediately from Proposition~\ref{prop: hpb=pb} and Proposition~\ref{prop: rel secat=secat hpb}.

\begin{corollary}\label{cor: secat of fibration}
    If $p$ is a $G$-fibration, $\gsc{G,f}{p}=\sct_G(f^\ast p).$
\end{corollary}

\subsection{Properties}
We study various bounds for the equivariant relative sectional category.
\begin{proposition}\label{prop: rel secat leq secat}
 Let $p\colon E\to B$  and $f\colon K\to B$ be $G$-maps. Then  $ \gsc{G,f}{p}\leq \gsc{G}{p}.$
\end{proposition}
\begin{proof}
Let $U\subseteq B$ be a $G$-invariant open set over which $p$ admits a local $G$-section; that is, there exists a $G$-map $s\colon U \to E$ such that $ps \simeq_G \iota_U$, where $\iota_U\colon U \hookrightarrow B$ is the inclusion. 

It is clear that $f^{-1}(U) \subseteq K$ is a $G$-invariant open set. Then, we define a $G$-map $$s \circ f\vert{}_{f^{-1}(U)}\colon f^{-1}(U) \to E,$$ which clearly satisfies
$$p \circ \left(s \circ f\vert{}_{f^{-1}(U)}\right) \simeq_G \iota_U \circ f\vert{}_{f^{-1}(U)} = f\vert{}_{f^{-1}(U)}.$$
Then, applying above argument to a $G$-invariant open cover, we get $ \gsc{G,f}{p}\leq \gsc{G}{p}.$
\end{proof}

\begin{proposition}\label{prop: restricted relative secat}
Let $p\colon E\to B$  and $f\colon K\to B$ be $G$-maps. Suppose $H$, $H'$ are closed subgroups of $G$ such that $K^H , E^H, B^H$ are $H'$-invariant. Then 
$$\sct_{H',f^H}(p^H) \leq \sct_{G,f}(p).$$
\end{proposition}

\begin{proof}
    Let $U \subseteq K$ be an open subset over which $p$ admits a $G$-equivariant local section $\sigma\colon U \to E$ relative to $f$; that is, $p \circ \sigma \simeq_G f|_U$.

    Let $V = U \cap K^H$. This is an open subset of $K^H$, and since $U$ is $G$-invariant (and thus $H'$-invariant) and $K^H$ is $H'$-invariant, $V$ is also $H'$-invariant. Because the $G$-equivariant map $\sigma$ sends $H$-fixed points to $H$-fixed points, it restricts to a well-defined $H'$-map $\sigma|_V\colon V \to E^H$. Consequently, $p^H \circ \sigma|_V \simeq_{H'} f^H|_V$. Thus, $\sigma|_V$ provides an $H'$-equivariant local section over $V$ for $p^H$ relative to $f^H$. 
    
    Since any $G$-invariant open cover of $K$ relative to $f$ restricts to an $H'$-invariant open cover of $K^H$ relative to $f^H$, the inequality follows.
\end{proof}
In particular, we have the following inequalities.
\begin{corollary} 
    Let $p\colon E\to B$  and $f\colon K\to B$ be $G$-maps. Then
    \begin{enumerate}
    \item $\sct_{H',f^H}(p^H) \leq \sct_{G,f}(p)$ for every closed subgroup $H'$ of $ N(H).$ 
        \item $\sct_{f^H}(p^H) \leq \sct_{G,f}(p)$ for every closed subgroup $H$ of $G$.
        \item $\sct_{H',f}(p) \leq \sct_{G,f}(p)$ for every closed subgroup $H'$ of $G$.
    \end{enumerate}
\end{corollary}
\begin{proof}
    For the first part, we show that for every $G$-space $X$, $X^H$ is $H'$-invariant whenever $H'$ is a subgroup of the normalizer $ N(H).$
    If $x\in X^H$, we show that $h'x\in X^H$ for all $h'\in H'$. For $h\in H$, 
    we have \begin{align*}
        h(h'x)&=h'(h'^{-1}hh'x)\\
        &=h'x 
    \end{align*}
    The second equality follows because  $x\in X^H$ and $h'^{-1}hh'\in H$ by the definition of the normalizer.
Therefore, $K^H, E^H, B^H$ are $H'$-invariant, so the hypothesis of Proposition~\ref{prop: restricted relative secat} is satisfied, yielding the desired inequality.

The second part follows directly from part~(1) by taking $H'=\{e\}$, since $\{e\}\subseteq N(H)$.

The third part similarly follows  from part~(1) by taking $H=\{e\}$, since $H'\subseteq N(\{e\})=G$.
\end{proof}
The equivariant LS-category of equivariant maps was introduced in \cite{D-GC}. This notion generalizes the classical notion of equivariant LS-category of $G$-spaces.

\begin{definition}[{\cite[Definition 4.2]{D-GC}}]
    Let $f\colon X\to Y$ be a $G$-map. The \emph{equivariant LS-category} of $f$, denoted by $\cat_G(f)$, is defined as the smallest integer $n$ such that there exists a $G$-invariant open cover $\{U_0,\dots, U_n\}$ of $X$ such that for each $0\leq i\leq n$, the restriction $f|_{U_i}$ is $G$-homotopic to a $G$-map which takes values in an orbit of a point.
\end{definition}

We note that the equivariant LS-category of a map satisfies $$ \cat_G(f) \leq \min\{\cat_G(X), \cat_G(Y)\}.$$
Observe that, if $f=\mathrm{id}_X$, we have  $\cat_G(f)=\cat_G(X)$.

 The equivariant LS-category provides a lower bound for the equivariant relative sectional category under specific conditions, as shown below.
 \begin{proposition}\label{prop: secat lb by cat}
 Suppose $p\colon E\to B$  and $f\colon K\to B$ are $G$-maps such that $E$ is $G$-contractible. Then \[\cat_G(f)\leq \sct_{G,f}(p).\]
\end{proposition}
\begin{proof}
Suppose $U$ is a $G$-invariant open subset of $K$ over which there exist a $G$-map $s\colon U\to E$ such that $p\circ s\simeq_G f|_{U}$. Since $E$ is $G$-contractible, there  exists a $G$-homotopy $F\colon E\times I\to E$ such that $F(-,0)=\mathrm{id}_E$ and $F(-,1)=c$, where $c\colon E\to E$ is a $G$-map which takes values into a single $G$-orbit.
By pre-composing this homotopy with $s$ and post-composing it with $p$, we obtain a $G$-homotopy showing that the map $p \circ s\colon U \to B$ is $G$-homotopic to $p \circ c \circ s$.
Thus, we have $f\vert{}_U \simeq_G p \circ c \circ s$. 

Because $p$, $c$, and $s$ are all $G$-maps, and $c$ takes values within a single orbit in $E$, the composition $p \circ c\circ s$ necessarily takes values within a single $G$-orbit in $B$. Then, applying the same argument of $G$-invariant open cover of $K$, we obtain the desired result.
\end{proof}

 We now state the lemma \cite[Lemma 3.13]{Colman_Grant_Eqtc} about the conservation of isotropy.
 
\begin{lemma}\label{lem: conservation of isotropy}
    Let $X$ be a $G$-connected $G$-space with $G_x \subseteq G_y$ for some $x,y \in X$, then there exists a $G$-homotopy $F\colon \mathcal O_x \times I \ra X$ such that $F_0=\iota_{\mathcal O_x}$ and $F_1(\mathcal O_x) \subseteq \mathcal O_y,$ where $\iota_{\mathcal O_x}: {\mathcal O_x}\hookrightarrow X$ denotes the inclusion of the orbit.
\end{lemma}
In particular, if $y \in X$ is a fixed point (so $G_y = G$), then for every $x \in X$, there exists a homotopy that deforms the orbit $\mathcal{O}_x$ into the point $y$.

We also have an upper bound given by the equivariant LS-category under specific hypotheses.
\begin{proposition}\label{prop: secat ub by cat}
    Let $p\colon E\to B$  and $f\colon K\to B$ be $G$-maps.
    If one of the following assumptions holds: \begin{enumerate}
        \item $p(E^H)=B^H$ for all closed subgroup $H\leq G$,
        (or) 
        \item $B$ is $G$-connected with $E^G\neq \emptyset,$ then
        $$\gsc{G,f}{p}\leq \cat_G(f) \leq \min\{\cat_G(K), \cat_G(B)\}.$$
    \end{enumerate} 
    
\end{proposition}
\begin{proof}
   Let $U \subseteq K$ be a $G$-invariant open set such that $f|_U$ is $G$-homotopic (via a homotopy $F$) to a $G$-map $c\colon U\ra B$ with $c(U )\subseteq \mathcal O_b $ for some $b\in B$. To prove the bound, it is sufficient to construct a map $s\colon U \to E$ such that $p \circ s \simeq_G f\vert{}_U$.

\textbf{Case (1):} Because $b \in B^{G_b}$, the surjectivity assumption $p(E^{G_b}) = B^{G_b}$ guarantees the existence of an element $e \in E^{G_b}$ such that $p(e) = b$. In particular, $G_b \subseteq G_e$. 
Since $c(U)\subseteq \mathcal{O}_b$, there exists $g\in G$ such that $c(x) = gb$.
We define a $G$-map $s\colon U \to E$ by setting $s(x) = ge$ whenever $c(x) = gb$ for some $g \in G$. This map is well-defined because: if $g_1b=g_2b$, then $g_2^{-1}g_1\in G_b \subseteq G_e$ and hence $g_1e=g_2e$; preventing ambiguity from the choice of $g$. 
Finally, $p(s(x)) = p(ge) = g \cdot p(e) = gb = c(x)$. Since $c \simeq_G f\vert{}_U$, it follows that $p \circ s \simeq_G f\vert{}_U$.

\textbf{Case (2):} Let $e_0 \in E^G$, then its image $p(e_0) = b_0$ belongs to $B^G$. Using Lemma~\ref{lem: conservation of isotropy}, there exists a $G$-homotopy $F'\colon \mathcal{O}_b \times I \to B$ such that $F'(-, 0) = i_{\mathcal{O}_b}$ and $F'(\mathcal{O}_b, 1) = \{b_0\}$.
Define $F''\colon U \times I \to B$ by $F''(x, t) = F'(c(x), t)$. Then $F''(-, 0) = c$ and $F''(-, 1) = b_0$. 
Concatenating the homotopy $F$ with $F''$ provides a $G$-homotopy from $f\vert{}_U$ to the constant map $c_{b_0}\colon U \to \{b_0\} \subseteq B$. If we define the $G$-map $s\colon U \to \{e_0\} \subseteq E$, we have $p \circ s = c_{b_0} \simeq_G f\vert{}_U$.
\end{proof}
\begin{corollary}\label{cor: contractibility secat=cat}
      Let $p\colon E\to B$  and $f\colon K\to B$ be $G$-maps, where $E$ is a $G$-contractible space. If $B$ is $G$-connected with  $E^G\neq \emptyset$, or if $p(E^H)=B^H$ for all closed subgroups $H\leq G$ then $\sct_{G,f}(p)=\cat_G(f).$  
\end{corollary}

\begin{corollary}\label{cor: cat=secat of one point inclusion}
    Let $Y$ be a $G$-connected space.
    Suppose  $f\colon X\to Y$ is a $G$-map and $Y^G\neq \emptyset$. Then for $z\in Y^G$, we have $$\sct_{G,f}(\{z\}\hookrightarrow{} Y)=\cat_G(f).$$
\end{corollary}
\begin{proof}
    Let $p$ be the inclusion map $p\colon  \{z\} \hookrightarrow Y$. Applying Proposition~\ref{prop: secat lb by cat} and Proposition ~\ref{prop: secat ub by cat}(2), we obtain the result.
\end{proof}

We will now establish the product inequality for the equivariant relative sectional category.

\begin{proposition}\label{prop: secat prod ineq}
Let $G$ be a compact Lie group and $ K_i$ be completely normal $G$-spaces for $i=1,2$. If $p_i\colon E_i\ra B_i$ and $f_i\colon K_i \ra B_i$ are $G$-maps for $i=1,2$, then 
$f_1 \times f_2$ and $p_1\times p_2$ can be regarded as $G$-maps, and  $$\sct_{G,f_1\times f_2}(p_1\times p_2) \leq \gsc{G,f_1}{p_1}+\gsc{G,f_2}{p_2}.$$
\end{proposition}
\begin{proof}
If we consider the standard $G$-homotopy pullback ${P_{f_i,p_i}} $ of $f_i,p_i$ for $i=1,2$, then we have the following $G$-homotopy pullback 
\[\begin{tikzcd}
	{P_{f_1,p_1}}\times {P_{f_2,p_2}} & {E_1\times E_2} \\
	{K_1\times K_2} & {B_1\times B_2}
	\arrow[from=1-1, to=1-2]
	\arrow["{ \tilde p_1 \times  \tilde p_2}",from=1-1, to=2-1]
	\arrow["{p_1\times p_2}", from=1-2, to=2-2]
	\arrow["{f_1 \times f_2}"', from=2-1, to=2-2].
\end{tikzcd}\]
We have $\gsc{G}{ \tilde p_1\times  \tilde p_2} \leq \gsc{G}{ \tilde p_1} + \gsc{G}{ \tilde p_2} $ 
from \cite[Proposition 2.11]{arora2024equivariant}. Now the conclusion follows from Proposition~\ref{prop: rel secat=secat hpb}.
\end{proof}
The following proposition provides a cohomological lower bound, generalizing \cite[Theorem 5.15]{Colman_Grant_Eqtc} and \cite[Proposition 3.8(5)]{Hopf_inv_Gon_Luc_Grant}.

\begin{theorem}\label{prop: secat lb by coh}
    Let $h_G^\ast$ be a multiplicative equivariant cohomology theory. Let $f\colon  K \to B$ and $p\colon  E \to B$ be $G$-maps. If there exist classes $\alpha_1, \dots, \alpha_k \in \ker p^\ast$ such that $f^\ast(\alpha_1 \smile \cdots \smile \alpha_k) \neq 0$, then
    $$ \gsc{G,f}{p} \geq k. $$
\end{theorem}

\begin{proof}
    For $1 \leq i \leq k$, let $\beta_i = f^\ast(\alpha_i) \in h_G^\ast(K)$. 
    Following the equivariant homotopy pullback diagram given in Proposition~\ref{prop: rel secat=secat hpb}, where $\tilde{p}: P_{f,p}  \to K$ and $\tilde{f}: P_{f,p}  \to E$, we have 
    $$ \tilde{p}^\ast(\beta_i) = \tilde{p}^\ast f^\ast(\alpha_i) = \tilde{f}^\ast p^\ast(\alpha_i) = 0, $$
    since $\alpha_i \in \ker p^\ast$. Because $f^\ast$ is a ring homomorphism, the product is non-zero:
    $$ \beta_1 \smile \cdots \smile \beta_k = f^\ast(\alpha_1 \smile \cdots \smile \alpha_k) \neq 0. $$

    We note that the arguments presented in \cite[Theorem 2.7]{arora2024equivariant} apply to any multiplicative equivariant cohomology theory, not exclusively Borel cohomology. 

    For the sake of contradiction, suppose that $\gsc{G,f}{p} < k$. By Proposition~\ref{prop: rel secat=secat hpb}, this implies $\gsc{G}{\tilde{p}} < k$. Thus, there exists a $G$-invariant open cover $\{U_1, \dots, U_k\}$ of $K$ admitting $G$-equivariant local sections $\sigma_i\colon  U_i \to P_{f,p}$ such that $\tilde{p} \circ \sigma_i \simeq_G j_i$, where $j_i\colon  U_i \hookrightarrow K$ denotes the inclusion map. 
   Then we obtain
    $$ j_i^\ast(\beta_i) = \sigma_i^\ast \tilde{p}^\ast(\beta_i) = \sigma_i^\ast(0) = 0. $$
    
    The long exact sequence of cohomology for the pair $(K, U_i)$ yields an element $\gamma_i \in h_G^\ast(K, U_i)$ such that $q_i^\ast(\gamma_i) = \beta_i$, where $q_i\colon  K \hookrightarrow (K, U_i)$ is the inclusion of pairs. 
    
    Taking the cup product of these relative classes, we get
    $$ \gamma_1 \smile \cdots \smile \gamma_k \in h_G^\ast\left(K, \bigcup_{i=1}^k U_i\right) = h_G^\ast(K, K) = 0. $$
    By the naturality of the cup product, we must have
    $$ \beta_1 \smile \cdots \smile \beta_k = q^\ast(\gamma_1 \smile \cdots \smile \gamma_k) = q^\ast(0) = 0, $$
    where $q\colon  K \hookrightarrow (K, K)$. This directly contradicts the fact that $\beta_1 \smile \cdots \smile \beta_k \neq 0$. Therefore, we must have $\gsc{G,f}{p} \geq k$.
\end{proof}
\subsection{Behavior of relative sectional category}
We study the behavior of the equivariant relative sectional category under various changes of domains and codomains. The corresponding non-equivariant results can be found in \cite{Calcines_relTC}.
\begin{proposition}\label{prop: homotopy invariance for replacements and prop: secat_ff' p leq secat_f p}
\begin{enumerate}
    \item Suppose the following diagram is a $G$-homotopy commutative diagram and $\alpha, \beta, \gamma$ are $G$-homotopy equivalences
\[\begin{tikzcd}
	K & B & E \\
	{K'} & {B'} & {E'}.
	\arrow["f", from=1-1, to=1-2]
	\arrow["\alpha"', from=1-1, to=2-1]
	\arrow["\beta", from=1-2, to=2-2]
	\arrow["p"', from=1-3, to=1-2]
	\arrow["\gamma", from=1-3, to=2-3]
	\arrow["{f'}"', from=2-1, to=2-2]
	\arrow["{p'}", from=2-3, to=2-2]
\end{tikzcd}\]
Then $\gsc{G,f}{p}=\gsc{G,f'}{p'}$.
    \item Suppose the following diagram commutes up to $G$-homotopy
    \[\begin{tikzcd}[row sep=0.5em, column sep=2em]
	& B & E \\
	K \\
	& {B'} & {E'}.
	\arrow["\beta", from=1-2, to=3-2]
	\arrow["p"', from=1-3, to=1-2]
	\arrow["\gamma", from=1-3, to=3-3]
	\arrow["f", from=2-1, to=1-2]
	\arrow["{f'}"', from=2-1, to=3-2]
	\arrow["{p'}", from=3-3, to=3-2]
\end{tikzcd}\]
Then $\gsc{G,f}{p}\geq\gsc{G,f'}{p'}$.
    \item Suppose the following diagram commutes up to $G$-homotopy
\[\begin{tikzcd}[row sep=0.5em, column sep=2em]
	&& E \\
	K & B \\
	&& {E'}.
	\arrow["p"', from=1-3, to=2-2]
	\arrow["\gamma", from=1-3, to=3-3]
	\arrow["f", from=2-1, to=2-2]
	\arrow["{p'}", from=3-3, to=2-2]
\end{tikzcd}\]
Then $\gsc{G,f}{p}\geq\gsc{G,f}{p'}$.
In particular, for the $G$-maps 
\[\begin{tikzcd}
	K & B & E & {E'}
	\arrow["f", from=1-1, to=1-2]
	\arrow["p"', from=1-3, to=1-2]
	\arrow["{p'}"', from=1-4, to=1-3],
\end{tikzcd}\]
we have 
$\gsc{G, f}{p}\leq \gsc{G,f}{p\circ p'}.$
\item Suppose the following diagram commutes up to $G$-homotopy
\[\begin{tikzcd}[row sep=0.5em, column sep=2em]
	K \\
	& B & E. \\
	{K'}
	\arrow["f", from=1-1, to=2-2]
	\arrow["\alpha"', from=1-1, to=3-1]
	\arrow["p"', from=2-3, to=2-2]
	\arrow["{f'}"', from=3-1, to=2-2]
\end{tikzcd}\]
Then $\gsc{G,f}{p}\leq\gsc{G,f'}{p}$.
In particular,
for the $G$-maps
\[\begin{tikzcd}
	{K'} & K & B & E
	\arrow["{f'}", from=1-1, to=1-2]
	\arrow["f", from=1-2, to=1-3]
	\arrow["p"', from=1-4, to=1-3],
\end{tikzcd}\] we have 
$\gsc{G, f\circ f'}{p}\leq\gsc{G,f}{p}.$
\end{enumerate}
\end{proposition}
\begin{proof}
Since $\alpha, \beta, \gamma$ are $G$-homotopy equivalences, they admit $G$-homotopy inverses  $\alpha', \beta', \gamma'$ respectively. First, we show that $\gsc{G,f}{p}\leq\gsc{G,f'}{p'}$. Let $U'\subseteq K'$ be an $G$-invariant open subset over which $p'$ admits a $G$-equivariant local section $\sigma':U'\to E'$ relative to $f'$; that is, $p'\circ \sigma'\simeq_G f'|_{U'}$.

Define a $G$-map  $\sigma\colon \alpha^{-1}(U')\to E$ as the composition $\sigma=\gamma'\circ\sigma'\circ\alpha|_{\alpha^{-1}(U')}.$ 
From the $G$-homotopy commutativity of the diagram, 
\[\begin{tikzcd}
	K & B & E \\
	{K'} & {B'} & {E'} \\
	K & B & E
	\arrow["f", from=1-1, to=1-2]
	\arrow["\alpha"', from=1-1, to=2-1]
	\arrow["\beta", from=1-2, to=2-2]
	\arrow["p"', from=1-3, to=1-2]
	\arrow["\gamma", from=1-3, to=2-3]
	\arrow["{f'}"', from=2-1, to=2-2]
	\arrow["{\alpha'}"', from=2-1, to=3-1]
	\arrow["{\beta'}", from=2-2, to=3-2]
	\arrow["{p'}", from=2-3, to=2-2]
	\arrow["{\gamma'}", from=2-3, to=3-3]
	\arrow["f"', from=3-1, to=3-2]
	\arrow["p", from=3-3, to=3-2]
\end{tikzcd}\]
    it follows that \begin{align*}
        p\circ \sigma&=p\circ \gamma'\circ\sigma'\circ\alpha|_{\alpha^{-1}(U')}\\
        &\simeq_G (\beta'\circ p')\circ\sigma'\circ\alpha|_{\alpha^{-1}(U')}\\
        &\simeq_G \beta'\circ  f'|_{U'}\circ\alpha|_{\alpha^{-1}(U')} && (\text{Since } p'\circ \sigma'\simeq_G f'|_{U'}) \\
         &\simeq_G \beta'\circ  (\beta\circ f|_{\alpha^{-1}(U')})  \\
          &\simeq_G f|_{\alpha^{-1}(U')} && (\text{Since } \beta'\circ \beta\simeq_G \mathrm{id}_B).
    \end{align*}
    Thus, $\sigma$ is the required local section of $p$ relative to $f$ over the $G$-invariant open subset $\alpha^{-1}(U')\subseteq K$; that is,  $p\circ \sigma\simeq_G f|_{\alpha^{-1}(U')}$. This proves that $\gsc{G,f}{p}\leq\gsc{G,f'}{p'}.$
    
    By applying same argument to the inverses $\alpha', \beta',\gamma'$, we get the reverse inequality $\gsc{G,f}{p}\leq\gsc{G,f'}{p'}$, yielding the equality.

    Since the proofs of the remaining parts follow the same arguments as those in \cite[Section 3]{Calcines_relTC}, adapted to the equivariant setting, we omit the proofs.
\end{proof}

\subsection{Ganea type characterization}
We now establish a characterization of the equivariant relative sectional category in terms of the existence of a global section of an associated $G$-fibration.
Given a $G$-map $p\colon E \to B$, we can produce an $(n+1)$-fold fibered join of $p$, defined by
$$\ast_B^n E = \left\{ (e_0, t_0, \dots, e_n, t_n) \mathrel{\Big\vert{}} e_i \in E, \, t_i \in [0,1], \, p(e_i) = p(e_j), \, \sum_{i=0}^n t_i = 1 \right\} {\Bigg/} \sim,$$
where the equivalence relation is generated by
$$(e_0, t_0, \dots, e_i, 0, \dots, e_n, t_n) \sim (e_0, t_0, \dots, e_i', 0, \dots, e_n, t_n).$$
We have a canonical map $\ast_B^np\colon  \ast_B^nE \to B$ given by
$$\left(\ast_B^np\right)([e_0, t_0, \dots, e_n, t_n]) = p(e_0) = \cdots = p(e_n).$$
Moreover, $\ast_B^np$ is a $G$-map where $\ast_B^nE$ is equipped with the diagonal $G$-action, defined as follows:
$$g \cdot [e_0, t_0, \dots, e_n, t_n] = [g \cdot e_0, t_0, \dots, g \cdot e_n, t_n].$$
As shown in \cite[Lemma 3.3]{Symm_grant}, if $p$ is a (Serre) $G$-fibration with fiber $F$, then $\ast_B^np$ is a (Serre) $G$-fibration with fiber ${\ast}^nF$, the $(n+1)$-fold join of the fiber $F$.
The proof outlining how the sectional category is related to the join is given by \v{S}varc \cite[Theorem 3]{Sva} and James \cite[Proposition 8.1]{James}. We follow the ideas of \cite[Proposition 3.4]{Symm_grant} for the equivariant relative sectional category.
\begin{proposition}\label{prop: Ganea char}
Let $G$ be a compact Lie group. Let $p\colon E \to B$ be a $G$-fibration, $K$ be a paracompact $G$-space, and $f\colon K \to B$ be a $G$-map. Then $\gsc{G,f}{p}$ is the least integer $n$ for which there exists a $G$-map $\sigma\colon  K \to \ast_B^n E$ such that $\left(\ast_B^np\right) \circ \sigma = f$.
\end{proposition}
\begin{proof}
First, assume that $\ast_B^np$ admits a $G$-section $\sigma\colon  K \to \ast_B^n E$ relative to $f$; that is, $\left(\ast_B^np\right) \circ \sigma = f$.
Let $t_i\colon  \ast_B^n E \to [0,1]$ be the coordinate functions for $0 \leq i \leq n$, which are continuous $G$-maps.
Let $U_i = \sigma^{-1}(t_i^{-1}((0,1]))$; that is, $U_i = \{ k \in K \mid t_i(\sigma(k)) > 0 \}$. These are $G$-invariant open subsets of $K$. Thus, we may define local $G$-sections $\sigma_i\colon  U_i \to E$ relative to $f$ from the formula $$\sigma(k)=[\sigma_0(k),t_0,\dots,\sigma_n(k),t_n].$$

Conversely, suppose $\gsc{G,f}{p} \leq n$. Then there exist $G$-invariant open subsets $U_0, \dots, U_n$ that cover $K$, along with $G$-maps $\sigma_i\colon  U_i \to E$ such that $p \circ \sigma_i = f\vert{}_{U_i}$ for $0 \leq i \leq n$.
By \cite[Lemma 3.2]{Symm_grant}, a paracompact $G$-space is a $G$-paracompact space. Therefore, we may take a $G$-equivariant partition of unity $h_0, \dots, h_n$ subordinate to the open cover $\{U_i\}_{i=0}^n$ of $K$. We can then define a $G$-section $\sigma\colon  K \to \ast_B^n E$ by the formula
$$\sigma(k) = [\sigma_0(k), h_0(k), \dots, \sigma_n(k), h_n(k)],$$
which satisfies $\left(\ast_B^np\right) \circ \sigma = f$.
\end{proof}
We now establish dimension-connectivity bounds on equivariant relative sectional category. We use Grant's ideas from \cite[Theorem 3.5]{Symm_grant}.

\begin{theorem}\label{thm: secat ub dim-conn}
Let $f\colon K \to B$ be a $G$-map between $G$-CW complexes, where $K$ has dimension at least $2$. If $p\colon E \to B$ is a Serre $G$-fibration such that the fiber of the $H$-fixed-point map $p^H\colon E^H\to B^H$ is $(s-1)$-connected, for all subgroups $H$ of $G$, then
$$\gsc{G,f}{p} < \frac{\mathrm{hdim}_G(K) +1}{s+1}.$$
\end{theorem}

\begin{proof}
By Proposition~\ref{prop: Ganea char}, it is sufficient to show that the canonical map $\ast_B^np\colon  \ast_B^nE \to B$ admits a $G$-section relative to $f$ whenever $n \geq \frac{\mathrm{hdim}_G(K) +1}{s+1}-1$.
The obstructions to the existence of such a section belong to the Bredon cohomology groups
$$H_G^{i+1}(K;\pi_i(\ast^n \mathscr{F})),$$
where $\pi_i(\ast^n \mathscr{F})$ denotes the equivariant local coefficient system on $K$. The value of this system on the fixed-point space $K^H$ is isomorphic to $\pi_i((\ast^n F)^H)$.
Since the fiber $F^H$ is $(s-1)$-connected, standard properties of the topological join imply that the $(n+1)$-fold join $(\ast^n F)^H$ is $((s+1)(n+1)-2)$-connected. Hence, all obstruction groups vanish whenever the connectivity meets or exceeds the dimension of the complex minus one; that is, whenever
$$\mathrm{hdim}_G(K) - 1 \leq (s+1)(n+1) - 2.$$
Rearranging this inequality yields
$$\frac{\mathrm{hdim}_G(K) +1}{s+1} \leq n+1,$$
which completes the proof.
\end{proof}
\section{Equivariant topological complexity of equivariant maps}\label{sec: eq tc of eq maps}
In this section, we introduce the notion of equivariant topological complexity of $G$-maps, which generalizes Colman and Grant's \cite{Colman_Grant_Eqtc} equivariant topological complexity of $G$-spaces.
\subsection{Equivariant LS-category of equivariant maps}
We now study the properties of the Lusternik--Schnirelmann category for $G$-maps.
We have the product inequality for this invariant.
\begin{proposition}\label{porp: cat prod ineq}
Suppose $f\colon X\to Y$ and $f'\colon X'\to Y'$ are $G$-maps with both $Y$, $Y'$ are $G$-connected and $X\times X'$ is completely normal $G$-space. If $Y^G$ and $Y'^G$  are nonempty, then
    \[
    \cat_G(f\times f')\leq \cat_G(f)+\cat_G(f')
    .\]
\end{proposition}
\begin{proof}
    The proof of the above theorem follows from Proposition~\ref{prop: secat prod ineq} and Corollary~\ref{cor: cat=secat of one point inclusion}. 
\end{proof}
The relationship between the equivariant LS-category of a $G$-space and the LS-category of its orbit space is established in \cite{Eqlscategory}. These results can be extended to a $G$-map.
\begin{proposition}
    Let $f\colon  X \to Y$ be a $G$-map and let $f/G\colon  X/G \to Y/G$ be the induced map on the quotient spaces. Then 
    $$\cat_G(f) \geq \cat(f/G).$$
\end{proposition}
\begin{proof}
    Let $U \subseteq X$ be a $G$-invariant open subset such that $f|_U \simeq_G c_{y_0}$, where $c_{y_0}: U \to Y$ denotes a $G$-map that sends $U$ into a single orbit $\mathcal{O}_{y_0}$ of a point ${y_0} \in Y$. 

    Passing to the quotient spaces, this equivariant homotopy induces a homotopy between the restricted map $(f/G)|_{U/G}$ and the constant map $c\colon  U/G \to Y/G$, which is defined by $c([x]) = [{y_0}]$ for all $[x] \in U/G$. 
    
    Consequently, any $G$-categorical open cover of $X$ relative to $f$ descends to a categorical open cover of $X/G$ relative to $f/G$. This implies that $\cat_G(f) \geq \cat(f/G)$.
\end{proof}
For free $G$-actions, we have the following equality relating the equivariant LS-category of a $G$-map with the LS-category of the corresponding quotient map. 
\begin{theorem}\label{prop: free action cat/G}
    Let $G$ be a compact Lie group acting freely on $G$-spaces $X$ and $Y$, where $X$ is metrizable. Let $f\colon  X \to Y$ be a $G$-map. Then $\cat_G(f) = \cat(f/G)$.    
\end{theorem}
\begin{proof}
    By the previous proposition, we already have $\cat_G(f) \geq \cat(f/G)$. It remains to prove the reverse inequality.
    
    Let $U \subseteq X/G$ be an open subset such that the restricted map $(f/G)|_U$ is homotopic to a constant map $c\colon  U \to Y/G$, which sends every point of $U$ to a single point $[y_0] \in Y/G$. Let $F\colon  U \times I \to Y/G$ be such a homotopy. 
    
    Let $q\colon  X \to X/G$ be the quotient map. Define $V = q^{-1}(U)$, which is a $G$-invariant open subset of $X$. The restricted map $f|_V\colon  V \to Y$, given by the composition $V \hookrightarrow X \xrightarrow{f} Y$, is a $G$-equivariant lift of $(f/G)|_U$.

    Since $X$ and $Y$ are free $G$-spaces, the lifted homotopy ${F}$ preserves the orbits in the sense of \cite[Chapter I, Exercise 5]{Bredon1972}. 
    Since $X$ is metrizable, every subspace is also metrizable. Therefore, by the lifting theorem of Palais \cite[Chapter II, 7.3 Theorem]{Bredon1972}, the homotopy $F\colon  U \times I \to Y/G$ lifts to a $G$-equivariant homotopy $\tilde{F}\colon  V \times I \to Y$ starting at $\tilde{F}(-, 0) = f|_V$.
    
    At $t=1$, the map $\tilde{F}(-, 1)\colon  V \to Y$ is a $G$-equivariant lift of the constant map $F(-,1) = c$. Therefore, $\tilde{F}(-, 1)$ maps all of $V$ into the single orbit $\mathcal{O}_{y_0} = G \cdot y_0$ in $Y$.
    
    Thus, $f|_V$ is $G$-homotopic to a map whose image lies in a single orbit. Consequently, any open cover of $X/G$ relative to $f/G$ pulls back to a $G$-categorical open cover of $X$ relative to $f$, implying that $\cat_G(f) \leq \cat(f/G)$. This completes the proof of the equality.
\end{proof}

The following notion will be used to give a lower bound for the equivariant LS-category of a map.

\begin{definition}
Let $f\colon X\to Y$ be a $G$-map. The \emph{$G$-equivariant cup length} of $f$, denoted by $\cl_G(f)$, is the largest non-negative integer $k$ such that there exist cohomology classes $\alpha_1, \dots , \alpha_k\in \tilde h_G^\ast(Y)$ such that the image of cup products $f^\ast(\alpha_1 \cdots \alpha_k)\neq 0$.
\end{definition}
Note that, if $f $ is identity, $\cl_G(f)$ reduces to $\cl_G(Y).$ Also, for a trivial group action $\cl_G(f)$ reduces to the non-equivariant $\cl(f).$ 
\begin{proposition}\label{prop: cat geq cup}
    Let $h_G^\ast$ be a multiplicative equivariant cohomology theory, and let  $f\colon  X\ra Y$ be a $G$-map where $Y^G$ is nonempty. Then $\cat_G(f)\geq \cl_G(f)$.
\end{proposition}
\begin{proof}
    The proof follows from Corollary~\ref{cor: cat=secat of one point inclusion} and Theorem~\ref{prop: secat lb by coh}.
\end{proof}
We have a dimension-connectivity upper bound for the invariant.
\begin{proposition}\label{prop: cat leq dim con}
Let $G$ be a compact Lie group, and let $X$ be a $G$-CW complex of dimension at least 2 with a non-empty fixed point. Let $f\colon X \ra Y$ be a $G$-map between $G$-CW complexes where $Y^H$ is $s$-connected, for every subgroup $H$ of $G$, then 
   $$ \cat_{G}{(f)}< \frac{\mathrm{hdim}_G(X) +1}{s+1}.$$
\end{proposition}
\begin{proof}
    The proof follows from Corollary~\ref{cor: cat=secat of one point inclusion} and \cref{thm: secat ub dim-conn}.
\end{proof}
\subsection{Equivariant TC of equivariant maps (Scott's version)}\label{subsec: eq tc scotts version}
In this subsection, we introduce the equivariant analogue of  Scott's notion of topological complexity of maps.
\begin{definition}
  Let $f\colon X\to Y$ be a $G$-map. the equivariant topological complexity of $f$ (in the sense of Scott), denoted by $\TC_G(f)$, is defined as the smallest non-negative integer $n$ such that there exists a $G$-invariant open cover $\{U_0,\dots,U_n\}$ of $X\times X$ and $G$-maps $s_i\colon U_i\to Y^I$ satisfying $\pi_Y\circ s_i\simeq_G (f\times f)|_{U_i}$ for each $0\leq i\leq n$.
\end{definition}
Equivalently, we have $\TC_G(f):=\sct_{G,f\times f}(\pi_Y)$.
\begin{proposition}
    If $f,g\colon  X\ra Y$ are two $G$-homotopic maps, then $\TC_G(f)=\TC_G(g)$.
\end{proposition}
\begin{proof}
Follows from the following equalities
 $$\TC_G(f)=\sct_{G,f\times f}(\pi_Y)=\sct_{G,g\times g}(\pi_Y)=\TC_G(g).$$  
\end{proof}
\begin{remark} We record the following special cases.
\begin{enumerate}
    \item If $G$ acts trivially on $X$ and $Y$, then for a $G$-map $f\colon X\to Y$, we have $$\TC_G(f)=\TC(f).$$ 
    \item If $f=\mathrm{id}_{X}$, then we recover Colman-Grant's  equivariant topological complexity: $$\TC_G(f)=\TC_G(X).$$
\end{enumerate}
\end{remark}
\begin{proposition}
Let $f\colon X\to Y$ be a $G$-map with the following pullback diagram: $$
\xymatrix{
(f \times f)^\ast Y^I \ar[rr]^{\overline{f\times f}} \ar[d]_{(f \times f)^\ast {\pi_Y} } & & Y^I \ar[d]^{{\pi_Y}} \\
  X\times X \ar[rr]_{f\times f} & & Y \times Y.
}
$$   Then
$$\TC_G(f)=\sct_G((f \times f)^\ast {\pi_Y}).$$    
\end{proposition}
\begin{proof}
Since ${\pi_Y}$ is a $G$-fibration, the desired equality follows from Corollary \ref{cor: secat of fibration}.    
\end{proof}
The equivariant topological complexity of a map is bounded by the equivariant topological complexities of its domain and codomain.
\begin{proposition}\label{prop: scotts tc ineq tcX tcY}
Let $f\colon X\to Y$ be a $G$-map. Then 
\[
\TC_G(f)\leq \mathrm{min}\{\TC_G(X),\TC_G(Y)\}.
\]    
\end{proposition}
\begin{proof}
Since $\TC_G(f)=\sct_G((f \times f)^\ast {\pi_Y})$, the following inequality follows from Proposition~\ref{prop: rel secat leq secat}:
\[
\TC_G(f)=\sct_G((f \times f)^\ast {\pi_Y}) \leq \sct_G(\pi_Y)=\TC_G(Y).
\]
To show the inequality $\TC_G(f)\leq \TC_G(X)$, consider an open $G$-invariant subset $U$ of $X\times X$ with a $G$-section $s\colon U\to X^I$ of $\pi_X$. Define $s'\colon U\to Y^I$ by $s'(x,y)(t)=f (s(x,y)(t))$. Note that we have  $\pi_Y\circ s'(x,y)=(s'(x,y)(0),s'(x,y)(1))=(f(x),f(y))$ for $(x,y)\in U$. This proves the desired inequality.
\end{proof}
By replacing the map $f\colon  K \to B$ with $f \times f\colon  X\times X \to Y \times Y$ and $p$ with $\pi_Y$ in Proposition~\ref{prop: restricted relative secat}, we obtain the following result.
\begin{proposition}\label{prop: restricted Scott TC}
    Let $H$ and $H'$ be closed subgroups of $G$ such that $X^H$ and $Y^H$ are $H'$-invariant. Then for a $G$-map $f\colon X\to Y$,  $\TC_{H'}^{Sc}(f^H) \leq \TC_{G}^{Sc}(f)$.
\end{proposition}
\begin{corollary}\label{cor: restricted Scott TC}
    In particular, we have the following inequalities for a $G$-map $f\colon X\to Y$
    \begin{enumerate}
    \item $\TC_{H'}^{Sc}(f^H) \leq \TC_{G}^{Sc}(f)$ for every closed subgroup $H'$ of $ N(H).$ 
        \item $\TC(f^H) \leq \TC_{G}^{Sc}(f)$ for every closed subgroup $H$ of $G$.
        \item $\TC_{H'}^{Sc}(f) \leq \TC_{G}^{Sc}(f)$ for every closed subgroup $H'$ of $G$.
    \end{enumerate}
\end{corollary}
The following notion will be used to give a lower bound for the equivariant topological complexity of a map.
\begin{definition}
Let $f\colon X\to Y$ be a $G$-map. The \emph{$G$-equivariant zero-divisor cup length} of $f$, denoted by $\zcl_G(f)$, is the largest  non-negative integer $k$ such that there exist cohomology classes $\alpha_1, \dots , \alpha_k\in h_G^\ast(Y\times Y)$  satisfying $\Delta^\ast(\alpha_i)=0$ such that the image of cup products $(f\times f)^\ast(\alpha_1 \cdots \alpha_k)\neq 0$ where $\Delta\colon Y\to Y\times Y$ is the diagonal map.
\end{definition}

Note that, if $f $ is identity, $\zcl_G(f)$ reduces to $\zcl_G(Y).$ Also, for a trivial group action $\zcl_G(f)$ reduces to the non-equivariant $\zcl(f).$

\begin{proposition}\label{prop: lb ub for Scott TC}
Suppose $G$ is a compact Lie group and $f\colon X\to Y$ is a $G$-map. Then 
$$\zcl_G(f)\leq \TC_G(f).$$
Moreover, if $X$ and $Y$ are $G$-CW complexes where $X$ has dimension at least $2$ and $Y^H$ is $s$-connected for every subgroup $H$ of $G$, then 
$$
\TC_G(f)< \frac{2 \mathrm{hdim}_G(X)+1}{s+1}
.$$
\end{proposition}
\begin{proof}
The proof of the first part follows from the cohomological lower bound as in  Theorem~\ref{prop: secat lb by coh}.

We recall that $\mathrm{hdim_G}(X)=\min\{\dim(Z)\mid Z \text{ is a } $G$\text{-CW complex with } Z \simeq_G  X\}.$
Let $Z$ be a $G$-CW complex with minimal dimension, say $\dim Z=n$ such that $Z\simeq_G X.$
Then, $Z\times Z\simeq_G X\times X$ and $\dim(Z\times Z)=2\dim Z=2n.$
This implies $\mathrm{hdim}_G(X\times X)\leq 2n=2\mathrm{hdim}_G(X).$
Using \cref{thm: secat ub dim-conn}, we conclude the second part.
\end{proof}
\begin{proposition}
    Suppose $f,g\colon X\to Y$ are $G$-maps with $X\times X$ is completely normal. Then 
    \[
    \TC_G(f\times g)\leq \TC_G(f)+\TC_G(g).
    \]
\end{proposition}
\begin{proof}
    This follows from Proposition \ref{prop: secat prod ineq}.
\end{proof}
We now examine how the equivariant LS-category of a map provides bounds for its topological complexity.
\begin{proposition}\label{prop: scott cat ub lb}
  Let $f\colon X\to Y$ be a $G$-map.
  \begin{enumerate}
      \item If $H$ is a closed subgroup of $G$ with $X^H\neq \emptyset$, then  $\cat_H(f)\leq \TC_G(f).$
      \item If $Y$ is $G$-connected, then $\TC_G(f)\leq \cat_G(f\times f).$
  \end{enumerate}
\end{proposition}
\begin{proof}
    The second inequality follows from part (1) of Proposition \ref{prop: secat ub by cat}, since the $G$-map $\pi_Y$ satisfies $\pi_Y((Y^I)^H)=\pi_Y((Y^H)^I)=Y^H\times Y^H=(Y\times Y)^H$ for all closed subgroups $H\leq G$.

    Let $U\subseteq X\times X$ be an open subset equipped with a local section $s\colon  U \to Y^I$ such that $\pi_Ys=(f\times f) \iota$, where $\iota\colon  U \hookrightarrow X \times X$ is the inclusion.
    
    Let $z\in X^H$ then $f(z)\in Y^H$. We can construct two $H$-maps, $j\colon X\to X\times X$ and $j'\colon Y \to Y \times Y$, defined by $j(x)=(z,x)$ and $j'(y)=(f(z),y)$, respectively.
    Let $p$ be the $H$-equivariant pullback of $\pi_Y$ along $j'$. Here, $PY=\{\gamma\in Y^I\mid \gamma(0)=f(z)\}$ and $p(\gamma)=\gamma(1)$.
    We consider the following commutative diagram: 
\[
\begin{tikzcd}[row sep=2.0em, column sep=2.8em]
	{\color{blue} j^{-1}(U)} 
		&& {\color{red!80!black} U} \\
	& {\color{blue} PY} 
		&& {\color{red!80!black} Y^I} \\
	{\color{blue} X} 
		&& {\color{red!80!black} X\times X} \\
	& {\color{blue} Y} 
		&& {\color{red!80!black} Y\times Y.}	
	\arrow["{\iota|_{j^{-1}(U)}}"', hook, blue, from=1-1, to=3-1]
	\arrow[dashed, blue, from=1-1, to=2-2]
	\arrow[from=2-2, to=4-2, blue, "p"', description,
  pos=.4,
  every label/.append style={xshift=0pt}]
	\arrow["f"', blue, from=3-1, to=4-2]	
	\arrow["\iota", hook, red!80!black, from=1-3, to=3-3, description,
  pos=.7,
  every label/.append style={xshift=0pt}]
	\arrow["s", red!80!black, from=1-3, to=2-4]
	\arrow["{\pi_Y}", red!80!black, from=2-4, to=4-4]
	\arrow["{f\times f}"', red!80!black, from=3-3, to=4-4]	
	\arrow[
  from=1-1,
  to=1-3,
  black!70,
  "{j|_{j^{-1}(U)}}",
  description,
  pos=.5,
  every label/.append style={yshift=3pt}
]
	\arrow[black!70, from=2-2, to=2-4]
	\arrow["j"', black!70, from=3-1, to=3-3, description,
  pos=.6,
  every label/.append style={xshift=3pt}]
	\arrow["{j'}"', black!70, from=4-2, to=4-4]
\end{tikzcd}
\]
    Since $\pi_Y(sj|_{j^{-1}(U)})=(f\times f) \iota j|_{j^{-1}(U)}=(f\times f) j\iota|_{j^{-1}(U)}=j'f\iota|_{j^{-1}(U)}$, the universal property of the pullback $p\colon PY \to Y$ gives the existence of a $H$-map $\sigma:j^{-1}(U)\to PY$ satisfying the property $p\sigma=f\iota|_{j^{-1}(U)}=f|_{j^{-1}(U)}.$ Thus, $\gsc{H,f}{j}\leq \TC_G(f)$. As $PY$ is $H$-contractible, we have $\cat_H(f)\leq\gsc{H,f}{j}$ by  Proposition \ref{prop: secat lb by cat}. Combining the last two inequality, we get  $\cat_H(f)\leq\TC_H^{Sc}(f).$
    Using Corollary~\ref{cor: restricted Scott TC}(3), we obtain $\cat_H(f)\leq\TC_G(f).$
\end{proof}
\begin{corollary}\label{cor: cat TC relation_contractibility}
   Let $X$ and $Y$ be $G$-spaces such that $X^G\neq\emptyset$, $X\times X$ is completely normal, and $Y$ is $G$-connected. Let $f\colon X\to Y$ be a $G$-map. Then
    \begin{enumerate}
        \item $\cat_G(f)\leq \TC_G(f)\leq 2 \cat_G(f)$.
        \item $f$ is $G$-contractible if and only if $\TC_G(f)=0$.
    \end{enumerate}
\end{corollary}
\begin{proof}
    The proof of the first part is immediate from Proposition \ref{prop: scott cat ub lb} and Proposition \ref{porp: cat prod ineq}. Since $G$-contractibility is defined by a vanishing $G$-category, the second part immediately follows from the first.    
\end{proof}
The equivariant topological complexity decreases under the composition of maps.
\begin{proposition}\label{prop: scotts tc compositions ineq}
  Let $f\colon X\to Y$ and $g\colon Y\to Z$ be $G$-maps. Then
  \[
  \TC_G(g\circ f)\leq \mathrm{min}\{\TC_G(f), \TC_G(g)\}.
  \]
\end{proposition}
\begin{proof}
We note that $ \TC_G(g\circ f)\leq \TC_G(g)$ follows from Proposition \ref{prop: homotopy invariance for replacements and prop: secat_ff' p leq secat_f p}(4).
Let $U\subseteq X\times X$ admits a $G$-section of $\pi_Y$ relative to the map $f\times f$; that is, $\pi_Y\circ s\simeq_G (f\times f)|_U.$ Let $g^I\colon Y^I\to Z^I$ be a $G$-map given by $g^I(\gamma)(t)=g(\gamma(t))$, where $\gamma\in Y^I$ and $t\in I$.
From the commutative diagram,
\[\begin{tikzcd}
	&& {Y^I} & {Z^I} \\
	U & {X\times X} & {Y\times Y} & {Z\times Z}
	\arrow["{g^I}", from=1-3, to=1-4]
	\arrow["{\pi_Y}"', from=1-3, to=2-3]
	\arrow["{\pi_Z}", from=1-4, to=2-4]
	\arrow["s", dashed, from=2-1, to=1-3]
	\arrow[hook, from=2-1, to=2-2]
	\arrow["{f \times f}"', from=2-2, to=2-3]
	\arrow["{g \times g}"', from=2-3, to=2-4]
\end{tikzcd}\]
It follows that 
\begin{align*}
    \pi_Z\circ g^I \circ s &=(g \times g)\circ \pi_Y\circ s\\
    &\simeq_G (g \times g)\circ (f \times f)| _U \\
    &= (g\circ f) \times (g \circ f)| _U.
\end{align*}
Therefore,  
 $ \TC_G(g\circ f)\leq \TC_G(f)$. 
  This completes the proof.
\end{proof}
\begin{proposition}\label{prop: TC f behaves with homotopy inv}
Let $f\colon X\to Y$ be a $G$-map.
    \begin{enumerate}
        \item If $f$ admits a right $G$-homotopy inverse, then $\TC_G(f)=\TC_G(Y)$.
        \item If $f$ admits a left $G$-homotopy inverse, then $\TC_G(f)=\TC_G(X)$.
        \item If $f$  $G$-homotopy equivalence, then $\TC_G(f)=\TC_G(Y)=\TC_G(X)$.
    \end{enumerate}
\end{proposition}
\begin{proof}
(1) Since $f$ admits a right $G$-homotopy inverse, there exists a  $G$-map $g\colon  Y\to X$ such that $f\circ g\simeq_G \mathrm{id}_Y$.
Therefore, we have 
\begin{align*}
    \TC_G(Y)= \TC_G(\mathrm{id}_Y)
    &=\TC_G(f\circ g)\\
    &\leq 
\TC_G(f) \qquad \text{ by Proposition~\ref{prop: scotts tc compositions ineq}}\\
&\leq \TC_G(Y) \qquad \text{ by Proposition~\ref{prop: scotts tc ineq tcX tcY}}.
\end{align*}
This yields $
\TC_G(f)= \TC_G(Y).$

(2) Suppose $g\colon Y\to X$ is the left  $G$-homotopy inverse of $f$. That is, $g\circ f\simeq_G \mathrm{id}_X$. Then again using Proposition~\ref{prop: scotts tc ineq tcX tcY} and Proposition~\ref{prop: scotts tc compositions ineq}, we have the following inequality $$\TC_G(X)=\TC_G(g\circ f)\leq 
\TC_G(g)\leq \TC_G(X).$$

(3) This follows from parts (1) and (2).
\end{proof}
\begin{proposition}
Suppose we have the following commutative diagram of $G$-maps such that 
\begin{enumerate}
\item $h$ admits a right $G$-homotopy inverse 
$$\xymatrix{
X \ar[rr]^{h} \ar[dr]_{f} & & {X'} \ar[dl]^{g}  \\
 & {Y}.   
}$$ 
 Then $\TC_G(f)=\TC_G(g)$.  

 \item  $h$ admits a left $G$-homotopy inverse
$$\xymatrix{
&X \ar[dr]^{f} \ar[dl]_{g}   \\
Y \ar[rr]^{h}  & & {Y'}.  
}$$ 
 Then $\TC_G(f)=\TC_G(g)$.
 \end{enumerate}
\end{proposition}
\begin{proof}
(1) From Proposition~\ref{prop: scotts tc compositions ineq}, we have $\TC_G(f)=\TC_G(g\circ h)\leq \TC_G(g)$. 
Furthermore, since $h$ admits a right $G$-homotopy inverse, there is a $G$-map $h'\colon X'\to X$ such that $h\circ h'\simeq_G \mathrm{id}_{X'}$. Therefore, $\TC_G(g)=\TC_G(f\circ h')\leq \TC_G(f)$.

(2) The proof is similar to part (1).
\end{proof}
The following theorem generalizes \cite[Theorem 5.11]{Colman_Grant_Eqtc}. In contrast to their proof, we cannot apply \cref{prop: free action cat/G}, since the quotient $G$-space $G/H$ is free only when $H=e$. We therefore give a direct argument. As a consequence, our proof does not require the metrizability assumption on the group $G$.
\begin{theorem}
    Let $G$ be a path-connected topological group, and let $H$ be a closed normal subgroup of $G$. Assume $G$ acts on itself and on $G/H$ by left translation. Then, for the quotient map $f\colon  G \to G/H$, we have
    $$ \TC_{G}^{Sc}(f) = \cat(f)=\TC(f).$$
\end{theorem}
\begin{proof}
     Since $G^{\{e\}}=G\neq \emptyset$, Proposition~\ref{prop: scott cat ub lb}(1) gives
     the lower bound $\TC_{G}^{Sc}(f) \geq \cat(f)$.
    
    On the other hand, $G/H$ is $G$-connected, and hence we have $\TC_{G}^{Sc}(f) \leq \cat_G(f \times f)$  by Proposition~\ref{prop: scott cat ub lb}(2). We will prove directly that $\cat_G(f \times f) \leq \cat(f)$ by explicitly constructing the required equivariant homotopies.

    Let $\cat(f) = n$. Then there exists an open cover $\{V_0, \dots, V_n\}$ of $G$ and homotopies $F_i\colon  V_i \times I \to G/H$ such that $F_i(-, 0) = f|_{V_i}$ and $F_i(-, 1)$ is a constant map sending $V_i$ to some point $z_i H \in G/H$. 

    We consider the map $\alpha\colon  G \times G \to G$ defined by $\alpha(x, y) = x^{-1}y$. We note that $\alpha$ is constant on $G$-orbits because $\alpha(gx, gy) = (gx)^{-1}(gy) = x^{-1}y$. 
    For each $i$, define $$U_i = \alpha^{-1}(V_i) = \{ (x, y) \in G \times G \mid x^{-1}y \in V_i \}.$$ Since  $\alpha$ is continuous and constant on $G$-orbits, each $U_i$ is a $G$-invariant open subset of $G \times G$, and $\{U_i\}_{i=0}^n$ forms a $G$-invariant open cover of $G \times G$.

    For each $U_i$, we construct a $G$-equivariant homotopy $\tilde{F}_i\colon  U_i \times I \to G/H \times G/H$ by setting
    $$ \tilde{F}_i((x, y), t) = \left( xH, \, x \cdot F_i(x^{-1}y, t) \right). $$
    
    We verify that the continuous map  $\tilde{F}_i$ has the required properties.
    \begin{enumerate}
        \item \textbf{Equivariance.} For any $g \in G$, we have 
        \begin{align*}
            \tilde{F}_i(g(x, y), t) &= \tilde{F}_i((gx, gy), t) \\
            &= \left( gxH, \, gx \cdot F_i((gx)^{-1}gy, t) \right) \\
            &= \left( gxH, \, gx \cdot F_i(x^{-1}y, t) \right) \\
            &= g \cdot \tilde{F}_i((x, y), t).
        \end{align*}
        Thus, $\tilde{F}_i$ is a $G$-homotopy.
        
        \item \textbf{Initial Map.} At $t = 0$, since $F_i(-, 0) = f|_{V_i}$, we have
        \begin{align*}
            \tilde{F}_i((x, y), 0) &= \left( xH, \, x \cdot f(x^{-1}y) \right) \\
            &= \left( xH, \, x (x^{-1}y H) \right) \\
            &= (xH, yH) \\
            &= (f \times f)(x, y).
        \end{align*}
        
        \item \textbf{Final Map.} At $t = 1$, since $F_i(-, 1)$ is the constant map with value in $z_iH$, we have
        $$ \tilde{F}_i((x, y), 1) = (xH, x z_i H) = x \cdot (eH, z_i H). $$
        This shows that $\tilde{F}_i(-, 1)$ maps all of $U_i$ into the single $G$-orbit generated by the point $(eH, z_i H) \in G/H \times G/H$.
    \end{enumerate}

    Because $(f \times f)|_{U_i}$ is $G$-homotopic to a map into a single orbit for each set in the $G$-invariant cover, we conclude that $\cat_G(f \times f) \leq n = \cat(f)$. 
    
    Combining our inequalities, we have $\cat(f) \leq \TC_G(f) \leq \cat_G(f \times f) \leq \cat(f)$, forcing equality throughout.
\end{proof}
We now produce a map version of \cite[Proposition 5.12]{Colman_Grant_Eqtc}, which relates equivariant topological complexity with equivariant LS-category for a homomorphism between topological groups.
\begin{theorem}
    Let $f\colon  X \to Y$ be a $G$-equivariant continuous group homomorphism between topological groups. Assume $G$ acts on $X$ and $Y$ via topological group automorphisms, and that $Y$ is $G$-connected. Then $\TC_G(f) = \cat_G(f)$.
\end{theorem}
\begin{proof}
    First, we note that because $G$ acts on $X$ via group automorphisms, the identity element $e_X \in X$ is fixed by $G$. Hence, its isotropy group is $G_{e_X} = G$. 
    Similarly, for the identity element $e_Y$  in $Y$, we have $G_{e_Y} = G$ and hence the orbit $\mathcal{O}_{e_Y} = \{e_Y\}$. 
    Since $X$ contains a $G$-fixed point, we can apply Proposition~\ref{prop: scott cat ub lb}(1) to obtain the lower bound $\TC_{G}^{Sc}(f) \geq \cat_G(f)$.

    For the reverse inequality, assume $\cat_G(f) = n$. Then there exists a $G$-invariant open cover $\{U_0, \dots, U_n\}$ of $X$ and $G$-equivariant homotopies $F_i\colon  U_i \times I \to Y$ such that $F_i(-, 0) = f|_{U_i}$, and $F_i(-, 1)$ maps $U_i$ into a single $G$-orbit. Since $Y$ is $G$-connected and has a fixed point $e_Y$, by Lemma~\ref{lem: conservation of isotropy}, this target orbit can be chosen to be the fixed-point orbit $\mathcal{O}_{e_Y} = \{e_Y\}$. 

    For each $i \in \{0, \dots, n\}$, define the set
    $ V_i = \{ (x, y) \in X \times X \mid xy^{-1} \in U_i \}. $
    Because the group operations (multiplication and inversion) are continuous and the $G$-action is via automorphisms, the map $(x, y) \mapsto xy^{-1}$ is continuous and $G$-equivariant. Consequently, each $V_i$ is a $G$-invariant open subset of $X \times X$, and the collection $\{V_0, \dots, V_n\}$ forms an open cover of $X \times X$.

    We now construct a $G$-equivariant local section $\sigma_i\colon  V_i \to Y^I$ relative to $f \times f$ over each $V_i$ by taking
    $$ \sigma_i(x, y)(t) = F_i(xy^{-1}, t) \cdot f(y). $$
    We note that \begin{align*}
        \sigma_i(gx, gy)(t) &=F_i((gx)(gy)^{-1}, t) \cdot f(gy)\\
        &=F_i(g(xy^{-1}), t) \cdot f(gy) \quad  \text{(Since } G \text{ acts on X via group automorphism)}\\
        &=gF_i(xy^{-1}, t) \cdot gf(y)\qquad  \text{(Since } F_i, f \text{ are } G\text{-map)}\\
        &=g\cdot(F_i(xy^{-1}, t) \cdot f(y)) \quad \text{(Since } G \text{ acts on Y via group automorphism)}\\
        &=  g\cdot\sigma_i(x, y)(t)
    \end{align*}
    Thus, $\sigma_i$ is $G$-equivariant. We evaluate the endpoints of the path. At $t = 0$, using the fact that $f$ is a homomorphism:
    $$ \sigma_i(x, y)(0) = F_i(xy^{-1}, 0) \cdot f(y) = f(xy^{-1}) \cdot f(y) = f(x)f(y)^{-1}f(y) = f(x). $$
    At $t = 1$:
    $$ \sigma_i(x, y)(1) = F_i(xy^{-1}, 1) \cdot f(y) = e_Y \cdot f(y) = f(y). $$
    Thus, the path $\sigma_i(x, y)$ connects $f(x)$ to $f(y)$. This implies $\pi_Y \circ \sigma_i = (f \times f)|_{V_i}$, successfully establishing the required local sections for topological complexity. 
    
    Therefore, $\TC_G(f) \leq n = \cat_G(f)$. Combining the  inequalities yields $\TC_G(f) = \cat_G(f)$.
\end{proof}

\subsection{Murillo-Wu's version}\label{subsec: eq tc MW version}
We have the following equivariant generalization of the relative sectional category introduced by Murillo-Wu \cite{M-W}.
\begin{definition}
Let $E \xra p B\xra f Y$ be $G$-maps. \emph{The equivariant relative sectional category} (in the sense of Murillo-Wu)   
    $\sct_{G,f}^{MW}{(p)}$ is defined to be the smallest integer $n$ for which there exists a $G$-invariant open cover $\{U_0,\dots, U_n\}$ of $B$ with $G$-maps $s_i\colon U_i\ra E$ for $i=0, \dots, n$ such that $fps_i\simeq_Gf|_{U_i}$.
    Clearly, $$\sct_{G,f}^{MW}{(p)}=\sct_{G,f}{(f\circ p)} .$$
\end{definition}
We note that $f\simeq_G f'$ and $p\simeq_G p'$ imply $\sct_{G,f}^{MW}{(p)}=\sct_{G,f'}^{MW}{(p')} .$
\begin{definition}
    Given a $G$-map $f\colon  X \to Y$, the \emph{equivariant topological complexity} of $f$ (in the sense of Murillo-Wu) is defined to be 
    $\TC_G^{MW}(f) = \sct_{G, f \times f}^{MW}(\pi_X)$, where $\pi_X\colon  X^I \to X \times X$ denotes the $G$-equivariant free path space fibration. Equivalently, $$\TC_G^{MW}(f) = \sct_{G, f \times f}((f\times f)\circ\pi_X).$$
\end{definition}
The map $\pi_X$  is a fibration replacement of the diagonal map $\Delta_X\colon X\to X\times X$ given by the commutative diagram
\[\begin{tikzcd}
	X & {X\times X} \\
	{X^I} & {X\times X,}
	\arrow["{\Delta _X}", from=1-1, to=1-2]
	\arrow["\simeq"', from=1-1, to=2-1]
	\arrow["{\mathrm{id}}", from=1-2, to=2-2]
	\arrow["{\pi_X}"', from=2-1, to=2-2]
\end{tikzcd}\]
where the map $X\to X^I$ is a $G$-homotopy equivalence sending a point to its corresponding constant path. Using Proposition~\ref{prop: homotopy invariance for replacements and prop: secat_ff' p leq secat_f p}(1), we have $$\sct_{G, f \times f}((f\times f)\circ\Delta_X)=\sct_{G, f \times f}((f\times f)\circ\pi_X).$$
Consequently, this yields the equivalent formulation $\TC_G^{MW}(f) = \sct_{G, f \times f}((f\times f)\circ\Delta_X).$

We follow the approach of Scott \cite[Theorem 3.4]{Sc} to establish the equivalence of the various notions of the topological complexity of a map.
\begin{proposition}\label{prop: Scott=MW}
    For a $G$-map $f\colon  X \to Y$, we have $\TC_G(f) = \TC_G^{MW}(f)$.  
\end{proposition}
\begin{proof}
    First, we show that $\TC_G(f) \geq \TC_G^{MW}(f)$.
    Let $U \subseteq X \times X$ be a $G$-invariant open set on which there exists a $G$-equivariant local section $\sigma\colon  U \to Y^I$ such that $\pi_Y \circ \sigma = (f \times f)|_U$.
    
    We can define a $G$-homotopy $F\colon  U \times I \to Y \times Y$ by 
    $$ F(x_1, x_2, t) = (\sigma(x_1, x_2)(t), f(x_2)) $$
    for $(x_1, x_2) \in U$. 
    Note that at $t = 0$, $F(x_1, x_2, 0) = (f(x_1), f(x_2))$. At $t = 1$, 
    $$ F(x_1, x_2, 1) = (f(x_2), f(x_2)) = (f \times f) \circ \Delta_X \circ s(x_1, x_2), $$
    where $s\colon  U \to X$ denotes the $G$-map given by the projection $s(x_1, x_2) = x_2$.
    Thus, we have $$(f \times f) \circ \Delta_X \circ s\simeq_G(f \times f)|_U.$$
    This establishes the required deformation, proving $\TC_G(f) \geq \TC_G^{MW}(f).$

    Next, we show that $\TC_G(f) \leq \TC_G^{MW}(f)$.
    Let $U \subseteq X \times X$ be a $G$-invariant open set on which there exists a $G$-equivariant map $s\colon  U \to X$ such that $$(f \times f) \circ \Delta_X \circ s \simeq_G (f \times f)|_U.$$ This implies we have a deformation of $(f \times f)|_U$ into the diagonal $\Delta(Y) \subseteq Y \times Y$ via a $G$-homotopy $F\colon  U \times I \to Y \times Y$.

    We can then define a local section $\sigma\colon  U \to Y^I$ by
    $$ \sigma(x_1, x_2)(t) = 
    \begin{cases}
        p_1 \circ F(x_1, x_2, 2t) \quad &\text{for } t \in [0, 1/2], \\  
        p_2 \circ F(x_1, x_2, 2 - 2t) \quad &\text{for } t \in [1/2, 1].
    \end{cases} $$
    Here, $p_i$ denotes the projection $Y\times Y\to Y$ on the $i$-th component for $i=1,2.$ 
    We note that $\sigma$ is a well-defined $G$-map because at $t = 1/2$, the boundary pieces agree: $F(x_1, x_2, 1)$ lies on the diagonal $\Delta(Y)$, meaning $p_1(F(x_1, x_2, 1)) = p_2(F(x_1, x_2, 1))$. 

    Moreover, evaluating the endpoints of the path yields $$\sigma(x_1, x_2)(0) =p_1(F(x_1, x_2, 0)) =p_1(f(x_1),f(x_2))=f(x_1)$$ and $$\sigma(x_1, x_2)(1) =p_2(F(x_1, x_2, 0)) = p_2(f(x_1),f(x_2))=f(x_2).$$ Thus, we have $\pi_Y \circ \sigma = (f \times f)|_U$. This proves $\TC_G(f) \leq \TC_G^{MW}(f)$, concluding the proof.
\end{proof}

\subsection{Examples}\label{subsec: examples}
We now use the various inequalities established in the preceding sections to compute the equivariant topological complexity of $G$-maps.

\begin{example}\normalfont{
    Let the circle group $G = S^1$ act on the odd-dimensional sphere $S^{2m+1} \subseteq \mathbb{C}^{m+1}$ (for $m \geq 1$) via scalar multiplication:
    $$ \xi \cdot (z_0, z_1, \dots, z_m) = (\xi z_0, \xi z_1, \dots, \xi z_m), $$
    where $\xi \in S^1$. This action yields the standard quotient map (the Hopf fibration)
    $$ f\colon  S^{2m+1} \to \mathbb{CP}^m. $$

    We note that the motion planning algorithm given in \cite[Theorem 8]{FarberTC} is naturally $G$-equivariant for this $S^1$-action, which gives $\TC_G(S^{2m+1}) \leq 1$. The explicit arguments for a cyclic group action can be found in \cite[Lemma 4.2]{G-V}. For a quick verification, consider the subsets
    $$ U_0 = \{ (x, y) \in S^{2m+1} \times S^{2m+1} \mid x \neq -y \} 
    \text{ and }
    U_1 = \{ (x, y) \in S^{2m+1} \times S^{2m+1} \mid x \neq y \}. $$
    These sets form a $G$-invariant open cover of $S^{2m+1} \times S^{2m+1}$. 
    
    The equivariant local section on $U_0$ is given by a map $\sigma_0\colon  U_0 \to (S^{2m+1})^I$, which sends a pair $(x, y)$ to the unique minimal geodesic from $x$ to $y$.

    On $U_1$, we define the local section $\sigma_1\colon  U_1 \to (S^{2m+1})^I$, which sends a pair $(x, y)$ to the concatenation of two paths: the unique minimal geodesic from $x$ to $-y$, followed by a path from $-y$ to $y$ defined by
    $$ t \mapsto -\cos(\pi t) y + \sin(\pi t) \frac{v(y)}{|v(y)|}, $$
    where $v$ is a non-vanishing $G$-equivariant vector field on the odd sphere $S^{2m+1}.$

    By Proposition~\ref{prop: scotts tc ineq tcX tcY}, we have the upper bound
    $ \TC_G(f) \leq \TC_G(S^{2m+1}) \leq 1. $
    Furthermore, because $f$ is not null-homotopic, by Corollary~\ref{cor: cat TC relation_contractibility}(2), its non-equivariant topological complexity satisfies $\TC(f) \geq 1$. By Corollary~\ref{cor: restricted Scott TC}(2), this provides the lower bound
    $\TC_G(f) \geq \TC(f) \geq 1. $
    Combining these inequalities, we conclude that $\TC_G(f) = 1$.}
\end{example}

\begin{example}\normalfont{
    Let $G=C_2$ act on the space $X=S^2\times S^2$ by $\xi\cdot (x,y)=(y,x)$ where $\xi$ is the generator of $C_2.$ Consider the induced quotient map $f\colon  X\to X/C_2$, which can be regarded as a $G$-map with the trivial action on the codomain. It is known that the symmetric product space $\mathrm{SP_2}(S^2)\cong \mathbb{CP}^2$ (see, for example, \cite{Blagojevic02}), thus $Y=\mathbb{CP}^2$.
    
    To obtain a lower bound for the equivariant LS-category, we consider the multiplicative cohomology theory $h_G^\ast(-)=H^\ast(-/G,\mathbb Q)$.
    Since the induced map on the quotient $f/G\colon  X/G \to Y$ is identified to the identity map on $\mathbb{CP}^2$, the induced cohomology homomorphism $(f/G)^\ast$ is the identity on $H^\ast(\mathbb{CP}^2; \mathbb Q)$. 

    The rational cohomology ring is $H^\ast(\mathbb{CP}^2; \mathbb Q) \cong \mathbb Q[a]/(a^3)$ with $\vert{}a\vert{} = 2$. Since $a^2 \neq 0$, the non-equivariant cup-length is $\cl(X/G) = 2$. Consequently, the relative equivariant cup-length is $\cl_G(f) = 2$.

    Because $Y$ has non-empty fixed points (the action is trivial, so $Y^G = Y$), Proposition~\ref{prop: cat geq cup} yields the lower bound $\cat_G(f) \geq 2$. By Proposition~\ref{prop: cat leq dim con}, we also have the upper bound $\cat_G(f) \leq 2$, forcing $\cat_G(f) = 2$.

    The fixed-point set of $X$ is the diagonal $X^G = \{(x, x) \mid x \in S^2\} \cong S^2$. Because $Y$ is $G$-connected and $X$ has non-empty fixed points, Corollary~\ref{cor: cat TC relation_contractibility}(1) provides the upper bound: $\TC_G(f) \leq 2 \cat_G(f) = 4.$

    To establish a lower bound for the topological complexity, we examine the kernel of the diagonal map $\ker(\Delta_Y^\ast)$, which is generated by the class $\alpha = a \otimes 1 - 1 \otimes a$. A direct binomial expansion yields
\begin{align*}
    \alpha^4 &=a^4\otimes 1-4a^3\otimes a+6a^2\otimes a^2-4a\otimes a^3+1\otimes a^4  \quad (\text{Since } |a|= \text{even}) \\ 
&=6a^2 \otimes a^2 \neq 0 \quad (\text{Since } a^3=0 \text{ and } a^2\neq 0) .
\end{align*}
We now analyze the induced quotient map $(f \times f)/G$. Let $\Delta G = \{(g, g) \mid g \in G\}$ denote the diagonal subgroup of $G \times G$. The map $(f \times f)/G$ can be viewed as the canonical projection:
$$(X \times X)/\Delta G \to (X/G) \times (X/G) \cong (X \times X)/(G \times G)\cong ((X\times X)/\Delta G)/L,$$
where $L = (G \times G)/\Delta G$. It is known that the rational cohomology of finite group quotients, $H^\ast(Z/L; \mathbb Q) \cong H^\ast(Z; \mathbb Q)^L$. Thus, the induced map $((f \times f)/G)^\ast$ corresponds to the natural inclusion of the $L$-invariant cohomology classes. 

Because $((f \times f)/G)^\ast$ is an injective map in rational cohomology, the non-zero class $\alpha^4$ pulls back to a non-zero class: $((f \times f)/G)^\ast(\alpha^4) \neq 0$. By Proposition~\ref{prop: lb ub for Scott TC}, this non-vanishing zero-divisor cup-length strictly implies $\TC_G(f) \geq 4$. 
Combining the upper and lower bounds, we conclude that $\TC_G(f) = 4$.}
\end{example}

\begin{example}
\normalfont{
    Consider the cyclic group $G = C_2$ acting on the sphere $S^n\subseteq \mathbb{R}^{n+1}$ via the involution
    $$ \xi \cdot (x_0, x_1, \dots, x_n) = (-x_0, x_1, \dots, x_n), $$
    where $\xi$ is the generator of $C_2$. Let $f\colon S^n \to S^n$ be a $G$-map with $n\geq 1$ and  $\deg(f)\neq 0$.

    By \cite[Theorem 3.5]{Marzzntowicz_coh_sphere95}, we have $\deg(f^G)=\deg(f) \pmod 2.$ So $\deg(f^G)\neq0$ .
 
    For any map $h:S^n\to S^n$ of non-zero degree, recall from  \cite[Example 2.7]{M-W} that $$\TC(h) = \TC(S^n)=\begin{cases}1 & n=\text{odd},\\2
 & n=\text{even}.        
    \end{cases}$$
    
    Consequently, we have $\TC(f) = \TC^{}(S^{n}).$ 
    Moreover, since the fixed-point set $(S^n)^G \cong S^{n-1}$, we also  have $\TC(f^G) = \TC^{}(S^{n-1})$. 

    By Corollary~\ref{cor: restricted Scott TC}(2), we have the lower bound
    $\max \{ \TC(f), \TC(f^G) \} \leq \TC_G(f). $
    Since $n$ and $n-1$ have opposite pairities, exactly one of $\TC(S^n)$ and $\TC(S^{n-1})$ equals to $1$ while the others equal to $2$. Thus, $2=\max\{1,2\}=\max\{ \TC(S^n),\TC(S^{n-1})\} \leq \TC_G(f)$.

    On the other hand, applying the upper bound from Proposition~\ref{prop: lb ub for Scott TC}, we obtain
    $$ \TC_G(f) < \frac{2n + 1}{(n-1) + 1} = 2 + \frac{1}{n}. $$
    Since $\TC_G(f)$ must be an integer, this strictly implies $\TC_G(f) \leq 2$. Combining this with the lower bound, we conclude that
    $ \TC_G(f) = 2.$
    We note that, when $n$ is odd, $\TC(f)=1$ is strictly smaller than the corresponding equivariant counterpart.
    }
\end{example}

\begin{example}
\normalfont{    Consider the $G=C_2$-action on $Y=S^3$ given by the reflection $$\xi(x_0,x_1,x_2,x_3)=(-x_0,x_1,x_2,x_3).$$
    Let $X=Y\times Y$ be equipped with the corresponding diagonal $G$-action, and  consider the projection map $f:X\to Y$ given by $f(y,y')=y.$

    We observe that the $G$-map $f$ admits a right inverse $h:Y\to X$ given by $h(y)=(y,y)$, so that  $fh=\mathrm{id}_Y.$
    Consequently, by Proposition~\ref{prop: TC f behaves with homotopy inv}(1), we have
    $ \TC_G(f)=\TC_G(S^3)=2$.
    This value strictly larger than its non-equivariant counterpart, $ \TC(f)=\TC(S^3)=1.$

Similarly, using  Proposition~\ref{prop: TC f behaves with homotopy inv}(2), we have
    $ \TC_G(h)=\TC_G(S^3)=2$,
   which is likewise strictly larger than the non-equivariant value $ \TC(h)=\TC(S^3)=1.$} 
\end{example}

\vspace{0.3cm}
\noindent \textbf{Acknowledgement:}
We would like to express our sincere gratitude to Professor García-Calcines for his careful reading of an earlier version of this paper and for his valuable comments and suggestions, which have helped us improve the paper substantially.
Navnath Daundkar gratefully acknowledges the support of the DST-INSPIRE Faculty Fellowship (Faculty Registration No.~IFA24-MA218), Department of Science and Technology, Government of India; the Industrial Consultancy and Sponsored Research (IC\&SR), Indian Institute of Technology Madras, through the New Faculty Initiation Grant (RF25261395MANFIG009294).  Bittu Singh is grateful to the Prime Minister's Research Fellowship (PMRF ID 1302644), Government of India, for his financial support.
\bibliographystyle{plain} 
\bibliography{ref_eqtc}

@article {G-V,
    AUTHOR = {Gonz\'alez, Jes\'us and Velasco, Maurilio},
     TITLE = {Complex-projective and lens product spaces},
   JOURNAL = {Bol. Soc. Mat. Mex. (3)},
  FJOURNAL = {Bolet\'in de la Sociedad Matem\'atica Mexicana. Third Series},
    VOLUME = {20},
      YEAR = {2014},
    NUMBER = {2},
     PAGES = {319--333},
      ISSN = {1405-213X,2296-4495},
   MRCLASS = {55R25 (55M30 55P15 57R42)},
  MRNUMBER = {3264620},
MRREVIEWER = {J.\ M.\ Boardman},
       DOI = {10.1007/s40590-014-0019-5},
       URL = {https://doi.org/10.1007/s40590-014-0019-5},
}

@article{M-W,
  title={Topological complexity of the work map},
  author={Murillo, Aniceto and Wu, Jie},
  journal={Journal of Topology and Analysis},
  volume={13},
  number={01},
  pages={219--238},
  year={2021},
  publisher={World Scientific}
}

@incollection {Colmaneqcat,
    AUTHOR = {Colman, Hellen},
     TITLE = {Equivariant {LS}-category for finite group actions},
 BOOKTITLE = {Lusternik-{S}chnirelmann category and related topics ({S}outh
              {H}adley, {MA}, 2001)},
    SERIES = {Contemp. Math.},
    VOLUME = {316},
     PAGES = {35--40},
 PUBLISHER = {Amer. Math. Soc., Providence, RI},
      YEAR = {2002},
      ISBN = {0-8218-2800-2},
   MRCLASS = {55M30 (55M35)},
  MRNUMBER = {1962151},
MRREVIEWER = {Jeffrey\ A.\ Strom},
       DOI = {10.1090/conm/316/05493},
       URL = {https://doi.org/10.1090/conm/316/05493},
}

@article {Angel-Colman-Grant-Oprea-Moritainv-eqcat,
    AUTHOR = {Angel, A. and Colman, H. and Grant, M. and Oprea, J.},
     TITLE = {Morita invariance of equivariant {L}usternik-{S}chnirelmann
              category and invariant topological complexity},
   JOURNAL = {Theory Appl. Categ.},
  FJOURNAL = {Theory and Applications of Categories},
    VOLUME = {35},
      YEAR = {2020},
     PAGES = {Paper No. 7, 179--195},
      ISSN = {1201-561X},
   MRCLASS = {55M30 (55R91)},
  MRNUMBER = {4066806},
MRREVIEWER = {Jose\ M.\ Garc\'ia Calcines},
}

@article {Clapp-Puppe,
    AUTHOR = {Clapp, M\'onica and Puppe, Dieter},
     TITLE = {Invariants of the {L}usternik-{S}chnirelmann type and the
              topology of critical sets},
   JOURNAL = {Trans. Amer. Math. Soc.},
  FJOURNAL = {Transactions of the American Mathematical Society},
    VOLUME = {298},
      YEAR = {1986},
    NUMBER = {2},
     PAGES = {603--620},
      ISSN = {0002-9947,1088-6850},
   MRCLASS = {55M30 (55P50)},
  MRNUMBER = {860382},
MRREVIEWER = {E.\ Fadell},
       DOI = {10.2307/2000638},
       URL = {https://doi.org/10.2307/2000638},
}

@incollection {Fadelleqcat,
    AUTHOR = {Fadell, E.},
     TITLE = {The equivariant {L}justernik-{S}chnirelmann method for
              invariant functionals and relative cohomological index
              theories},
 BOOKTITLE = {Topological methods in nonlinear analysis},
    SERIES = {S\'em. Math. Sup.},
    VOLUME = {95},
     PAGES = {41--70},
 PUBLISHER = {Presses Univ. Montr\'eal, Montreal, QC},
      YEAR = {1985},
      ISBN = {2-7606-0712-7},
   MRCLASS = {58E05},
  MRNUMBER = {801933},
MRREVIEWER = {Marco\ Degiovanni},
}

@article{D-GC,
  title={Equivariant homotopic distance},
  author={Daundkar, Navnath and Garc{\'\i}a-Calcines, J. M.},
  journal={arXiv preprint arXiv:2508.20485, Topological Methods in Nonlinear Analysis to appear},
  year={2026}
}

@article {James,
    AUTHOR = {James, I. M.},
     TITLE = {On category, in the sense of {L}usternik-{S}chnirelmann},
   JOURNAL = {Topology},
  FJOURNAL = {Topology. An International Journal of Mathematics},
    VOLUME = {17},
      YEAR = {1978},
    NUMBER = {4},
     PAGES = {331--348},
      ISSN = {0040-9383},
   MRCLASS = {55M30 (58E05)},
  MRNUMBER = {516214},
MRREVIEWER = {Christian Fenske},}

@article{LScat,
	author = {Lyusternik, Lazar and \v{S}nirelmann, Levi},
	volume = {42},
	publisher = {Presses Universitaires de France},
	title = {M\'{e}thodes Topologiques Dans les Probl\`{e}mes Variationnels},
	year = {1934},
	journal = {Actualit\'{e}s scientifiques et industrielles, vol. 188, Expos\'{e}s sur l’analyse math\'{e}matique et ses applications, vol.3, Hermann, Paris},
	pages = {},
	number = {}
}

@Article{secat,
 Author = {Berstein, I. and Ganea, T.},
 Title = {The category of a map and of a cohomology class},
 FJournal = {Fundamenta Mathematicae},
 Journal = {Fundam. Math.},
 ISSN = {0016-2736},
 Volume = {50},
 Pages = {265--279},
 Year = {1962},
 Language = {English},
 DOI = {10.4064/fm-50-3-265-279},
 zbMATH = {3305695},
 Zbl = {0192.29302}
}

@article {Eqlscategory,
    AUTHOR = {Marzantowicz, Wac{\l}aw},
     TITLE = {A {$G$}-{L}usternik-{S}chnirelman category of space with an
              action of a compact {L}ie group},
   JOURNAL = {Topology},
  FJOURNAL = {Topology. An International Journal of Mathematics},
    VOLUME = {28},
      YEAR = {1989},
    NUMBER = {4},
     PAGES = {403--412},
      ISSN = {0040-9383},
   MRCLASS = {55M30 (57S15 58E05)},
  MRNUMBER = {1030984},
MRREVIEWER = {S\"{o}ren Illman},}

@article {Sva,
    AUTHOR = {\v{S}varc, Albert S.},
     TITLE = {The genus of a fiber space},
   JOURNAL = {Dokl. Akad. Nauk SSSR (N.S.)},
  FJOURNAL = {Doklady Akademii Nauk SSSR},
    VOLUME = {119},
      YEAR = {1958},
     PAGES = {219--222},
      ISSN = {0002-3264},
   MRCLASS = {55.00},
  MRNUMBER = {0102812},
MRREVIEWER = {J. Stasheff},
}

@article {Symm_grant,
    AUTHOR = {Grant, Mark},
     TITLE = {Symmetrized topological complexity},
   JOURNAL = {J. Topol. Anal.},
  FJOURNAL = {Journal of Topology and Analysis},
    VOLUME = {11},
      YEAR = {2019},
    NUMBER = {2},
     PAGES = {387--403},
      ISSN = {1793-5253,1793-7167},
   MRCLASS = {55M30 (55P91 55R91 55S15 55S40 68T40)},
  MRNUMBER = {3958926},
MRREVIEWER = {Marzieh\ Bayeh},
       DOI = {10.1142/S1793525319500183},
       URL = {https://doi.org/10.1142/S1793525319500183},
}

@article {Colman_Grant_Eqtc,
    AUTHOR = {Colman, Hellen and Grant, Mark},
     TITLE = {Equivariant topological complexity},
   JOURNAL = {Algebr. Geom. Topol.},
  FJOURNAL = {Algebraic \& Geometric Topology},
    VOLUME = {12},
      YEAR = {2012},
    NUMBER = {4},
     PAGES = {2299--2316},
      ISSN = {1472-2747,1472-2739},
   MRCLASS = {55M30 (57S10)},
  MRNUMBER = {3020208},
MRREVIEWER = {Ian\ Hambleton},
       DOI = {10.2140/agt.2012.12.2299},
       URL = {https://doi.org/10.2140/agt.2012.12.2299},
}

@article {Calcines_relTC,
    AUTHOR = {Garc\'ia-Calcines, J. M.},
     TITLE = {Relative sectional category revisited},
   JOURNAL = {Expo. Math.},
  FJOURNAL = {Expositiones Mathematicae},
    VOLUME = {42},
      YEAR = {2024},
    NUMBER = {5},
     PAGES = {Paper No. 125590, 22},
      ISSN = {0723-0869,1878-0792},
   MRCLASS = {55M30 (55M15 55P99)},
  MRNUMBER = {4775632},
       DOI = {10.1016/j.exmath.2024.125590},
       URL = {https://doi.org/10.1016/j.exmath.2024.125590},
}

@article {Hopf_inv_Gon_Luc_Grant,
    AUTHOR = {Gonz\'alez, Jes\'us and Grant, Mark and Vandembroucq, Lucile},
     TITLE = {Hopf invariants for sectional category with applications to
              topological robotics},
   JOURNAL = {Q. J. Math.},
  FJOURNAL = {The Quarterly Journal of Mathematics},
    VOLUME = {70},
      YEAR = {2019},
    NUMBER = {4},
     PAGES = {1209--1252},
      ISSN = {0033-5606,1464-3847},
   MRCLASS = {55M30},
  MRNUMBER = {4045098},
MRREVIEWER = {Jose\ M.\ Garc\'ia Calcines},
       DOI = {10.1093/qmath/haz019},
       URL = {https://doi.org/10.1093/qmath/haz019},
}

@article {MACÍAS–VIRGÓS_MOSQUERA–LOIS_2022,
    AUTHOR = {Mac\'ias-Virg\'os, E. and Mosquera-Lois, D.},
     TITLE = {Homotopic distance between maps},
   JOURNAL = {Math. Proc. Cambridge Philos. Soc.},
  FJOURNAL = {Mathematical Proceedings of the Cambridge Philosophical
              Society},
    VOLUME = {172},
      YEAR = {2022},
    NUMBER = {1},
     PAGES = {73--93},
      ISSN = {0305-0041,1469-8064},
   MRCLASS = {55M30 (55P45 55R10 68T40)},
  MRNUMBER = {4354415},
MRREVIEWER = {Nicholas\ A.\ Scoville},
       DOI = {10.1017/S0305004121000116},
       URL = {https://doi.org/10.1017/S0305004121000116},
}

@article {Sc,
    AUTHOR = {Scott, Jamie},
     TITLE = {On the topological complexity of maps},
   JOURNAL = {Topology Appl.},
  FJOURNAL = {Topology and its Applications},
    VOLUME = {314},
      YEAR = {2022},
     PAGES = {Paper No. 108094, 25},
      ISSN = {0166-8641,1879-3207},
   MRCLASS = {55M30 (55R80)},
  MRNUMBER = {4406987},
MRREVIEWER = {Jose\ M.\ Garc\'ia Calcines},
       DOI = {10.1016/j.topol.2022.108094},
       URL = {https://doi.org/10.1016/j.topol.2022.108094},
}

@article{arora2024equivariant,
  title={Equivariant and invariant parametrized topological complexity},
  author={Arora, Ramandeep Singh and Daundkar, Navnath},
  journal={Proceedings of the Royal Society of Edinburgh: Section A Mathematics},
  year={2026},
  pages={1–44},
  DOI={10.1017/prm.2026.10147}
}

@article {FarberTC,
    AUTHOR = {Farber, Michael},
     TITLE = {Topological complexity of motion planning},
   JOURNAL = {Discrete Comput. Geom.},
  FJOURNAL = {Discrete \& Computational Geometry. An International Journal
              of Mathematics and Computer Science},
    VOLUME = {29},
      YEAR = {2003},
    NUMBER = {2},
     PAGES = {211--221},
      ISSN = {0179-5376},
   MRCLASS = {68T40 (55M30 68U05)},
  MRNUMBER = {1957228},
      }

@book{Bredon1972,
  author    = {Glen E. Bredon},
  title     = {Introduction to Compact Transformation Groups},
  series    = {Pure and Applied Mathematics},
  volume    = {46},
  publisher = {Academic Press},
  address   = {New York},
  year      = {1972},
  isbn      = {9780121297501}
}

@inproceedings{Blagojevic02,
  author    = {Blagojevi{\'c}, Pavle and Gruji{\'c}, Vladimir and
               {Z}ivaljevi{\'c}, Rade},
  title     = {Symmetric products of surfaces; a unifying theme for topology and physics},
  booktitle = {Proceedings of Summer School in Modern Mathematical Physics},
  editor    = {Dragovi{\'c}, B. and Sazdovi{\'c}, B.},
  series    = {SFIN},
  volume    = {XV},
  number    = {A3},
  publisher = {Institute of Physics},
  address   = {Belgrade},
  year      = {2002},
  pages     = {466--491}
}

@article {Marzzntowicz_coh_sphere95,
    AUTHOR = {Izydorek, Marek and Marzantowicz, Wac\l aw},
     TITLE = {Equivariant maps between cohomology spheres},
   JOURNAL = {Topol. Methods Nonlinear Anal.},
  FJOURNAL = {Topological Methods in Nonlinear Analysis},
    VOLUME = {5},
      YEAR = {1995},
    NUMBER = {2},
     PAGES = {279--289},
      ISSN = {1230-3429},
   MRCLASS = {55M25 (55M35 57S17)},
  MRNUMBER = {1374065},
MRREVIEWER = {Katsuhiro\ Komiya},
       DOI = {10.12775/TMNA.1995.019},
       URL = {https://doi.org/10.12775/TMNA.1995.019},
}
\end{document}